\documentclass[11pt]{amsart}
\usepackage{amscd}                                            
\usepackage[arrow,matrix,all]{xy}
\usepackage{graphicx}
\usepackage{hyperref}
\usepackage{comment}
\usepackage{amsmath}
\usepackage{color}
\usepackage{mathabx}
\usepackage[normalem]{ulem}
\usepackage{listings}
\usepackage{mathbbol}

\usepackage{amsmath, latexsym, amssymb}
\usepackage{enumitem}

\numberwithin{equation}{subsection}
\theoremstyle{plain}
\newtheorem{lemma}[equation]{Lemma}
\newtheorem{proposition}[equation]{Proposition}
\newtheorem{theorem}[equation]{Theorem}
\newtheorem{corollary}[equation]{Corollary}

\theoremstyle{definition}
\newtheorem{definition}[equation]{Definition}
\newtheorem{remark}[equation]{Remark}

\newtheorem{examples}[equation]{Examples}

\newtheorem{notation}[equation]{Notation}

\newtheorem*{proposition*}{Proposition}
\newtheorem*{lemma*}{Lemma}

\usepackage{cancel}

\definecolor{brown}{RGB}{150,100,0}
 
\definecolor{purple}{RGB}{150,0,100}

\definecolor{grey}{RGB}{128,128,128}

\newcommand{\R}{{\mathbb R}}

\renewcommand{\O}{{\mathbb O}}

\renewcommand{\H}{{\mathbb H}}

\newcommand{\Ll}{{\mathcal L}}    

\newcommand{\Rr}{{\mathcal R}}

\newcommand{\Om}{{\Omega}}
\newcommand{\om}{{\omega}}

\newcommand{\so}{{\mathfrak{so}}}

\newcommand{\g}{{\mathfrak g}}

\renewcommand{\Im} {{\mathrm {Im }\, }}

\newcommand{\vol}{{\mathrm {vol}}}

\newcommand{\Hom}{{\mathrm{Hom}}}

\def\NABLA#1{{\mathop{\nabla\kern-.5ex\lower1ex\hbox{$#1$}}}}
\def\Nabla#1{\nabla\kern-.5ex{}_#1}

\newcommand{\la}{\langle}
\newcommand{\ra}{\rangle}

\newcommand{\GL}{{\rm GL}}

\newcommand{\SO}{{\rm SO}}
\newcommand{\Sp}{{\rm Sp}}

\newcommand{\G}{{\mathrm G_2}}

\newcommand{\Spin}{{\rm Spin}}
\DeclareMathOperator{\Stab}{Stab}

\newcommand{\ver}{\mathrm{vert}}
\newcommand{\hor}{\mathrm{hor}}

\DeclareMathOperator{\Id}{Id}

\begin{document}
	\title[CR-twistor spaces  over manifolds with $\G$  -and $\Spin(7)$-structures]{CR-twistor spaces   over manifolds with $\G$  -and $\Spin(7)$-structures}
	
	\author{Domenico Fiorenza}
	\address{
		Dipartimento di Matematica ``Guido Castelnuovo'',
		Universit\`a   di Roma ``La Sapienza",
		Piazzale Aldo Moro 2, 00185 Roma, 
		Italy}
	\email{fiorenza@mat.uniroma1.it} 
	
	\author{H\^ong V\^an L\^e }
	\address{Institute  of Mathematics of the Czech Academy of Sciences,
		Zitna 25, 11567  Praha 1, Czech Republic}
		\email{hvle@math.cas.cz}
		\thanks{Research  of HVL was supported  by  GA\v CR-project 22-00091S  and
		 RVO: 67985840}
	\date{\today}


\begin{abstract}  In 1984  LeBrun constructed  a  CR-twistor  space  over  an arbitrary conformal Riemannian  3-manifold and proved that  the  CR-structure  is  formally integrable.   This   twistor  construction  has been    generalized by Rossi in 1985  for  $m$-dimensional Riemannian  manifolds endowed with a $(m-1)$-fold  vector cross product (VCP). In 2011 Verbitsky   generalized    LeBrun's construction   of   twistor-spaces     to   $7$-manifolds  endowed  with    a $\G$-structure.  In this paper we unify    and generalize     LeBrun's, Rossi's  and  Verbitsky's   construction of a CR-twistor  space to the case    where   a  Riemannian  manifold  $(M, g)$      has  a  VCP  structure. We show  that the  formal	 integrability of the CR-structure is expressed  in terms  of  a torsion tensor  on   the  twistor space, which  is a  Grassmanian bundle over $(M, g)$.  If  the VCP structure on $(M,g)$ is generated by a  $\G$- or $\Spin(7)$-structure,  then the vertical component of  the  torsion tensor  vanishes  if and only if  $(M, g)$ has constant curvature,  and the horizontal component  vanishes   if  and only  if $(M,g)$  is a  torsion-free $\G$ or $\Spin(7)$-manifold. Finally  we discuss some open problems.	 
 \end{abstract}

 \maketitle
  
\tableofcontents

\section{Introduction}\label{sec:intr}

\subsection{Motivations and prior  works}\label{subs:motivations}
  In  his papers \cite{BG1967}, \cite{Gray1969},  motivated  by Calabi's work on   almost complex structures  on $S^6$, Gray introduced the notion  of  a vector cross product (VCP for short) structure.
By definition, a $r$-fold VCP structure $\chi$  on an Euclidean vector space $(V,\la   \cdot, \cdot \ra)$ is a multilinear alternating map
	 \[
	 \chi\colon \bigwedge^{r}V \to V
	 \]
such that 	 
$$	\la \chi(v_1, \cdots, v_r), v_i  \ra = 0 \text{  for }   1\le i \le r,$$
$$	\la  \chi(v_1, \cdots, v_r), \chi(v_1, \cdots , v_r) \ra = \| v_1 \wedge \cdots \wedge v_r\|^2,$$
where $\| \cdot \|$ is  the induced metric  on $\wedge ^r V$.   
For a   $r$-fold VCP $\chi$ on $V$,  the associated  VCP-form $\varphi_\chi\colon \bigwedge^{r+1}V\to \mathbb{R}$  is defined as 
\begin{equation}\label{eq:varphi-chi}
\varphi_\chi   (v_1, \cdots,  v_{r+1}) = \la \chi(v_1, \cdots, v_r), v_{r+1}\ra,
\end{equation}
\cite[(4.1)]{Gray1969}. As a matter of notation, once a $r$-fold VCP is fixed one often writes $v_1\times v_2\times\cdots\times v_r$ for $\chi(v_1,\dots,v_r)$.
	 
\begin{remark}\label{rem:BG1}
(1) The Brown-Gray  classification \cite{BG1967} asserts  that a     $r$-fold  VCP  structure    exists  on $\R^m$ if and  only if one of the following  possibilities holds  (i) $r=1$ and  $ m$ is even;
	(ii)  $r =  m-1$;  (iii) $ r = 2$ and $ m = 7$;   (iv)  $ r =3$  and  $m = 8$.

(2) A $(m-1)$-fold  VCP   structure  $\chi$  on $(\R^m, \la \cdot, \cdot \ra)$ is defined   uniquely  by     a  given orientation on $\R^m$.

(3)	For $m =7$, $r=2$      the  VCP form $\varphi_\chi$ is called  the  {\it  associative 3-form}.  Its stabilizer in $\GL (\R^7)$ is   the  exceptional group  $\G\subset  \SO (7)$.  For
	$m = 8$, $r = 3$ the VCP form  $\varphi_\chi$ is called   the {\it Cayley  4-form}. Its  stabilizer  in $\GL (\R^8)$  is the subgroup  $\Spin(7) \subset \SO (8)$.
The   VCP  structures  $(\chi, g)$  in these  cases   are in  a 1-1 
 correspondence  with    their   VCP-forms  $\varphi_\chi$.    Given  a VCP form $\varphi$  on  a 7-manifold       an explicit  formula   for $g_\varphi$  is given  in  \cite[\S 7.1]{Hitchin2000}. Similarly,  given  a  VCP form $\varphi$  on  a 8-manifold,       a  formula   for $g_\varphi$       can be obtained   using     the relation $\varphi ^2 = 8 \vol_{g_\varphi}$  and  Hitchin's method, see  similar results in   \cite[\S 3]{LPV2008}.
\end{remark}

One has an immediate notion of a $r$-fold VCP on a Riemannian  manifold $(M, g)$ as  a smooth $TM$-valued   $r$-form $\chi \in \Om ^r(M, TM)$ such that $\chi(x)$ is a $r$-fold  VCP  on $T_xM$ for all $x \in M$.  The corresponding VCP-form will therefore be an element in  $\Om^{r+1} (M)$.

 \begin{remark}\label{rem:BG3}
 A VCP form  $\varphi_\chi$ is parallel w.r.t.  the Levi-Civita  connection $\nabla ^{LC}_g$  iff  either $(M^m, g)$ is  an orientable    Riemannian  manifold  and  $r = m-1$; or
$m =2 n$,
$(M^{2n}, g)$ is a K\"ahler manifold  and $r =1$; or    $m = 7$  and $(M^7, g)$ is a  torsion-free
$G_2$-manifold and $r =2$; or $m =8$  and  $(M^8, g)$  is a torsion-free   Spin(7)-manifold  and $r =3$. This result singles out 
K\"ahler  manifolds, torsion-free $G_2$-and Spin(7)-manifolds  as important  classes   of Riemannian manifolds with special holonomy \cite{Joyce2000}. Not unrelatedly,  these classes play  a prominent  role in  calibrated geometry, string theory and M-theory,  and  F-theory \cite{Joyce2007}, \cite{GS2002}, \cite{BGP2014}. 
\end{remark}

 \begin{remark}\label{rem:BG4}  The    VCPs  in dimension 3, 7, 8   can be expressed  in terms  of algebraic operations on normed  algebras. Denote by $\Im  \O$ the imaginary   part  of the  octonion algebra $\O$. Harvey and Lawson noticed   that, identifying $\R^7$ with  $\Im \O$,  the      associative 3-form $\varphi_\chi$ on $\Im \O$   has the following form
\cite[(1.1), p. 113]{HL1982}:
	\begin{equation}\label{eq:varphi}
	\varphi_\chi(x, y, z) =       \la  x, yz \ra.
	\end{equation}
Hence   the    2-fold  VCP  $\chi$ on $\Im  \O$ is defined as follows \cite[Definition B.1, p. 145]{HL1982} 
 \begin{equation}\label{eq:vcp2}
  y \times z = \Im  (yz ).
 \end{equation}
 The restriction  of this  2-fold     VCP  to  $\Im  \H \subset \Im  \O$  coincides  with the   2-fold
 VCP on $\R^3$  \cite[p. 145]{HL1982}.
 
   The  3-fold VCP  on $\R^ 8 = \O$   can be expressed as follows \cite[Definition B.3, p. 145]{HL1982}: 
\begin{equation}\label{eq:vcp3}
  u \times v \times  w =  {1\over 2}\left(( u \bar v) w - (w \bar v)u\right). 
  \end{equation}

\end{remark}

The relation between  complex  structures  and VCP  structures     has been manifested  also   via  CR-twistor spaces  over    manifolds  endowed with a  VCP structure.
 In 1984  LeBrun constructed  a  CR-twistor  space  over  an arbitrary conformal Riemannian  3-manifold \cite{LeBrun1984}.
 LeBrun   proved that the CR-twistor space of a conformal Riemmannian 3-manifold is  a  CR-manifold,   i.e. the  CR-structure  is integrable.   This   twistor  construction  has been    generalized by Rossi in 1985  for  $m$-dimensional Riemannian  manifolds endowed with a $(m-1)$-fold VCP \cite{Rossi1985}   and  utilized further  by LeBrun  for his proof  of the  formal  integrability  of the  almost complex structure $J$
 on     the higher dimensional  loop space over    a Riemannian manifold   $(M^m, g)$ endowed with  a $(m-1)$-fold VCP \cite{LeBrun1993}, following  a similar      proof by Lempert  for  the weak  integrability  of  the almost complex structure on the  loop space over a     Riemannian 3-manifold \cite{Lempert1993}.
 In 2011 Verbitsky   generalized    LeBrun's construction   of   twistor-spaces     to   $7$-manifolds  endowed  with    the VCP 3-forms $\varphi$ \cite{Verbitsky2011},  which subsequently has been used by him for    his proof of the formal integrability of the  almost complex  structure on  the loop space over   a  holonomy $\G$-manifold \cite{Verbitsky2010}. 
 
\subsection{Our main results}\label{subsec:main}

 As a first result in this   paper,  we      unify    and generalize     LeBrun's, Rossi's construction  of a CR-twistor
 space  over   a  conformal  Riemannian manfold  in dimension 3  and in arbitrary dimension respectively,  as well as  Verbitsky's   construction of a CR-twistor  space over  a $\G$-manifold to the case    when   the underlying   Riemannian  manifold  $(M, g)$      has  a  VCP  structure, see Definition  \ref{def:twistor}  below. In order to state the result we need fixing notation.

\begin{notation} 
Let $(M, g)$ be an oriented  Riemannian manifold.  

$\bullet$  We denote  by $\mathbb{Gr}^+ (r-1,M)$ the  Grassmannian  of
oriented  $(r-1)$-planes  in $TM$,  which  we  shall  identify  with  decomposable unit  $(r-1)$-vectors in $\bigwedge ^{r-1}TM$. When no confusion is possible we will denote $\mathbb{Gr}^+ (r-1,M)$ simply by $\mathbb{G}$.  We denote by $\pi:  \bigwedge ^{r-1} TM \to M$ the natural projection, which also induces    the natural projection $\pi:  {\mathbb{G}} \to M$.  For any point $v\in  {\mathbb{G}}$, the fiber of $\pi:  {\mathbb{G}} \to M$ through $v$ is naturally identified with the Grassmannian ${\mathbb{Gr}}^+ ({r-1}, T_{ \pi (v)} M)$ of oriented  $(r-1)$-planes  in $T_{ \pi (v)} M$.

$\bullet$ For $v \in  {\mathbb{G}}$ we denote by  $E_v\subseteq T_{\pi (v)}M$  the oriented ${(r-1)}$-plane  associated
to $v$  and  by  $E_v ^\perp$ its orthogonal complement  in  $T_{\pi (v)}M$.
\end{notation}
The  Riemannian metric $g$   induces a natural  Riemannian metric  on   the vector bundle $\bigwedge  ^{{r-1}} TM \stackrel{\pi}{\to}  M$      and so endows  $\bigwedge   ^{r-1} TM$  with   the corresponding Levi-Civita  connection $\nabla ^{LC}$. This induces,   for any $ v \in \bigwedge^{r-1} TM$,  a direct sum decomposition  
\begin{equation}\label{eq:hor1}
T_v (\wedge ^{r-1}TM) =  \wedge ^ {r-1} T_{\pi (v)}M \oplus  T_v^{\mathrm{hor}}(\wedge ^{r-1} TM), 
\end{equation}
where  
$$ T_v ^{\mathrm{hor}}(\wedge ^{r-1}TM)  \cong  T_{\pi (v)} M$$
is the horizontal  distribution in $T\wedge ^{r-1} TM$ w.r.t.      $\nabla ^{LC}$. 
Since $\mathbb{Gr} ^+ ({r-1}, TM)$  is a     fiber sub-bundle  of the vector bundle  $\bigwedge^{r-1} TM \stackrel{\pi}{\to} M$,  for $ v \in {\mathbb{G}}$  the orthogonal decomposition   \eqref{eq:hor1}   induces  the decomposition
\begin{equation}\label{eq:hor2}
T_v{\mathbb{G}} =   T_v ^{\mathrm{vert}}{\mathbb{G}} \oplus^\perp   T_v ^{\mathrm{hor}}{\mathbb{G}},
\end{equation}
where
\begin{equation}\label{eq:vert}
T_v ^{\mathrm{vert}}{\mathbb{G}} =  T_v  \mathbb{Gr}^+ ({r-1}, T_{ \pi (v)} M)
\end{equation}
and 
\begin{equation}\label{eq:hor3}
T_v ^{\mathrm{hor}}{\mathbb{G}}  =  T_v ^{\mathrm{hor}}(\wedge ^{r-1} TM) \cong T_{\pi (v)} M.
\end{equation}
 \begin{notation}Let $(M^m,g)$ be an $m$-dimensional Riemannian manifold. We denote by $B$ the rank $m-r+1$ distribution on ${\mathbb{G}}$  
  defined  at a  point $v$ of ${\mathbb{G}}$ by
\begin{equation}\label{eq:B}
B_v: = \{ w \in T_v ^{\mathrm{hor}}{\mathbb{G}}|\:    d\pi_v(w) \in E_v ^{\perp}  \subset  T_{\pi (v)} M\}\subseteq T_v ^{\mathrm{hor}}{\mathbb{G}}.
\end{equation}
An $r$-fold VCP structure $\chi$ on $(M, g)$ endows the vector spaces $E_v^\perp$ with a complex structure $J_{E_v^\perp}$ defined by
\begin{equation}\label{eq:Jev}
J_{E_v^\perp}(z)=\chi(v\wedge z),
\end{equation}
for $z\in E_v^\perp$, see \cite[Lemma 3.1]{FL2021}, \cite[p. 146]{LL2007},  \cite[Theorem 2.6]{Gray1969}. 
Since $d\pi_v: B_v \to E_v ^\perp$ is an isometry, the complex structure $J_{E_v^\perp}$ induces a complex structure $J_{g, \chi}$ on 
 $B_v$. It is defined by the equation
\begin{equation}\label{eq:BJv}
 d\pi_v (J_{g, \chi}  (w))  =  J_{E_v^\perp}( d\pi_v  (w)).
 \end{equation}
 \end{notation}
  
 \begin{definition} (cf. \cite[Definitions 1.1, 1.2,  p. 3]{DT2006})\label{def:CR}
An {\it almost CR-structure} on a manifold $N$ is a pair $(B,J_B)$ consisting of a distribution $B\subseteq TN$ and of an almost complex structure $J_B$ on $B$.  The triple $(N,B,J_B)$ is called an \emph{almost CR-manifold}.
An almost   CR-structure $(B,J_B)$ on  a  manifold $N$ is said to be {\it formally
integrable} if the complex distribution $B^{1,0}\subset B\otimes\mathbb{C}$  is involutive, i.e., $[B^{1,0},B^{1,0}]\subseteq B^{1,0}$. If $(B,J_B)$ is integrable, then the almost CR-manifold $(N,B,J_B)$ is called a {\it CR-manifold}.
\end{definition}

\begin{remark}\label{rem:integrabiliy-conditions}
 The condition that the  almost CR-structure $(B,J_B)$ is formally integrable can be stated completely in terms of sections of the real vector bundle $B$, without going through its complexification, as follows. Denote by  $\Pi_B$ the orthogonal projection  of $TGr^+ (r-1,  M)$  to   $B$ and  by $\Gamma (B)$ the space  of smooth sections   of $B$. 
\begin{enumerate}
\item \label{CR1} For any $X,Y\in \Gamma(B)$ one has $[J_BX,J_BY]-[X,Y]\in  \Gamma(B)$;

\item \label{CR2} For any $X,Y\in \Gamma(B)$ one has 
\[
 \Pi_B ([J_BX,J_BY]-[X,Y])-J_B\circ \Pi _B([X,J_BY]+[J_BX,Y])=0.
\]
\end{enumerate}

In  literature \cite[p. 128]{bejancu}, \cite[p. 4]{DT2006}    the condition  \eqref{CR2}   is   replaced  by the  following condition
$$	[J_BX,J_BY]-[X,Y]-J_B([X,J_BY]+[J_BX,Y])=0,$$
which has meaning  only  if the condition \eqref{CR1} holds. Clearly     the conditions  \eqref{CR1} and \eqref{CR2}  are equivalent  to the condition \eqref{CR1}  and  the  classical  condition  stated above.

We shall call  the  condition \eqref{CR1}  \emph{the first CR-integrability  condition}, and the  condition \eqref{CR2} \emph{the second    CR-integrability condition}.
In Cartan geometry,  the  condition  \eqref{CR1}  is also called  the {\it partial integability} of a CR-structure \cite[p. 443]{CS09}.
\end{remark}

Now  we    associate  to each  VCP-structure on a Riemannian  manifold $(M, g)$  an almost CR-manifold as follows.

\begin{definition}\label{def:twistor}  Let $(M, g, \chi)$  be a Riemannian manifold  endowed with a VCP structure $\chi$.   The almost CR-manifold $({\mathbb{G}}, B,J_{g,\chi})$  consisting of the manifold ${\mathbb{G}}$ together with the almost CR-structure given by the distribution  $B$ and the  almost complex  structure   $J_{g,\chi}$  on $B$ defined in Equations (\ref{eq:B}) and \eqref{eq:BJv}  will be called the {\it CR-twistor  space  over $(M, g, \chi)$}.	
\end{definition}    

\begin{examples}\label{ex:CRVCP}
(1) Let   $(\chi,g)$ be  a  1-fold VCP on a  smooth manifold  $M^{2n}$.  This is equivalently an Hermitian almost complex structure on $M$. In this case one has
${\mathbb{G}} = M^{2n}$,  the horizontal distribution $B$  is identified with the tangent bundle to  $M^{2n}$ and the  almost complex  structure $J_{g, \chi}$  is identified with the almost complex structure on $M$.

(2) Let  $(\chi, g)$  be a  $(m-1)$-fold    VCP  on  an oriented  manifold $M^m$. Then the   distribution  $B$ on ${\mathbb{G}}$ is  2-dimensional and the   CR-twistor    structure  $({\mathbb{G}}, B, J_{g, \chi}) $  coincides  with  the one constructed  by LeBrun \cite{LeBrun1984}  and  extended by  Rossi \cite{Rossi1985}.

(3) Let  $\chi$  be   a 2-fold  VCP on $(M^7,g)$.  Then    the   CR-twistor structure  on ${\mathbb{G}}$ coincides with   the one constructed  by  Verbitsky in \cite{Verbitsky2011}.
\end{examples}

 Our main result  in this  paper  concerns necessary and sufficient  conditions for the first and the second CR-integrability of the CR-twistor  space  over  a
 Riemannian manifold  $(M, g)$  endowed with  a VCP structure.    Let $\nabla_g ^{LC}$ denote  the  Levi-Civita  connection on $(M, g)$. We say that the VCP $\chi$ is parallel if $\nabla_g ^{LC}\chi=0$.
 
 \

In this   paper  we  prove the following

\begin{theorem}[Main Theorem]\label{thm:main}	
	Let  $\chi\in \Om^{r+1}(M, TM)$ be a VCP structure on a Riemannian manifold $(M, g)$ and  $({\mathbb{G}}, B, J_{g, \chi})$   the associated     CR-twistor  space.  Then  there  exists  a tensor  $T \in  \Gamma (\wedge ^2 B^*\otimes  T{\mathbb{G}})$ on the total space ${\mathbb{G}}$  such   that
	
	\begin{enumerate}
		\item  The first  CR-integrability \eqref{CR1} holds if and only if  for any  $ v \in {\mathbb{G}}$ and $X, Y \in B (v)$ we have   $T^\ver (X, Y) =0\in T^\ver {\mathbb{G}}$.
		
		\item If $(r,m)=(1,2n)$ or $(m-1,m)$ then  $T^\ver = 0$   for any  $(M, g, \chi)$.

		\item If $(r,m)=(2,7)$ or $(3,8)$ then  $T^\ver  = 0$ 
		if   and only if $(M, g)$    has constant  curvature.
	
		\item	 The  second  CR-integrability \eqref{CR2} holds, if and only if    for any  $ v \in {\mathbb{G}}$ and $X, Y \in B (v)$ we have   $T^\hor (X, Y) =0\in  T^\hor {\mathbb{G}}$.

		\item If $(r,m)=(1,2n)$ then  $T^\hor =0$  for $(M, g, \chi)$ 
		if and only if  $\chi$ is integrable.

		\item If $(r,m)=(m-1,m)$ then $T^\hor  = 0$  for any  $(M, g, \chi)$. 
		
		\item If $(r,m)=(2,7)$  or    $(r, m)  =  (3, 8)$ then  $T^\hor = 0$  
		if and only  if  $\chi$ is parallel.
	\end{enumerate}	
\end{theorem}
\begin{remark}\label{rem:parts-5-6} 
Parts (2\&5) and  (2\&6) of the main theorem above combined,  i.e., without decomposing the CR integrability condition into two independent conditions, are classical and we are including them only for completeness. 
In particular, by combining Example \ref{ex:CRVCP} (1) with Example \ref{ex:CR1LeBrun}(1) we recover that  a  Riemannian manifold $(M^{2n}, g)$  endowed with a 1-fold  VCP $\chi$ is
 a CR-manifold if and only if the almost complex structure on $M$  induced by $\chi$ is integrable.
Part (6) is due to  LeBrun, who proved that the   CR-twistor  space over  a  Riemannian manifold $(M^m, g)$  with a $(m-1)$ fold  VCP  $\chi$  is  always  a CR-manifold \cite{LeBrun1984}, \cite{LeBrun1993}. Note that in this case  $\chi$ is always parallel, see \cite[Proposition 4.5]{Gray1969}. 
Part (7) for the case $(2,7)$ is due to
Verbitsky \cite{Verbitsky2011}. Unfortunately his  proof  uses  a  wrong argument, see  Remark  \ref{rem:verprob}. 
Finally, it was also known that the   CR-twistor  space over any   flat  Riemannian manifold $(M, g)$ endowed  with  parallel  VCP is a  CR-manifold.
\end{remark}

\subsection{Organization of  our paper}\label{subs:org}
 In the second section   we   study   the  first  condition \eqref{CR1}   for the formal integrability of  the CR-twistor space
over  a Riemannian  manifold  $(M, g)$ endowed   with  a VCP-structure $\chi$. First, using  a geometric characterization   of   the     distribution $B$
 (Lemma \ref{lemma:annihilator}),   we express the condition \eqref{CR1} for  the CR-twistor space over $(M, g, \chi)$     in terms of the vertical components of the Lie brackets $[J_BX,J_BY]$  and $[X,Y]$,  where $X, Y \in \Gamma (B)$,  with respect to  the  decomposition \eqref{eq:hor2} (Corollary \ref{cor:cond2}). Using this,  we   prove  that  the first  CR-integrability condition  \eqref{CR1}  for   the CR-twistor   space over  $(M, g, \chi)$  holds if and only if   the  curvature  $R(g)$   of  the  underlying Riemannian manifold  $(M, g)$  is   a solution of 
an infinite system   of linear   equations   (Proposition  \ref{prop:ver32}).
  Next,   we  study  the first CR-integrability  condition for  the  $(2,7)$  case  using Proposition  \ref{prop:ver32}  and computer  algebra. Using  these results and ad-hoc methods in  Section \ref{sec:spin7},  we  prove  assertions (1)  and (3)  of  Theorem \ref{thm:main}. 
In Section \ref{sec:second} we  study   the  second     condition \eqref{CR2}  for the formal integrability  of the   CR-twistor space over $(M,g, \chi)$   using the  formalism  of the  Fr\"ohlicher-Nijenhuis  bracket  (Proposition \ref{prop:CR2}).  Then we give the proof of  Theorem  \ref{thm:main}  (4, 7).  Finally we  discuss   our results  and some open  questions. We  include an  Appendix containing {\tt sagemath} codes for solving   the first and  the second  CR integrability condition  in the (2, 7)  case.

\subsection{Notation and conventions}
\

$\bullet$ We  keep  notation in the introduction.   

$\bullet$ For a vector  bundle  $E$ over  a  manifold  $M$ and a smooth section  $\alpha \in \Gamma (E)$,    we  also write  $\alpha _x$ for the value  $ \alpha(x)$    to avoid  possibly  ugly notation  like $\alpha(x) (v)$ occurring, e.g, when $E$ is the endomorphism bundle of $TM$ and $v$ is a tangent vector at $x$.

$\bullet$  If $\xi$ is an   element in   a  vector  space $V$   with   inner product $\la, \ra$,  we denote by $\xi^\sharp$ the element  on $V^*$ defined  by $\xi^\sharp (v) = \la  \xi, v\ra$ for  all $v \in V$.

$\bullet$  Given a  $G$-action  on   a  space   $X$,  for  $x \in X$, we denote by $\mathrm{Stab}_G (x)$ the stabilizer of $x$ in $G$. 

$\bullet$    We   consider   in this paper   the Killing metric  on Lie algebra $\so (\R^n)$ and any its  Lie subalgebra  defined as follows   $\la X, Y\ra =  - \frac{1}{2} Tr  (XY)$.

$\bullet$
Let $(V,\langle,\rangle)$ be an Euclidean vector space. We denote by $\mathcal{AC}(V)$ the vector subspace of $\bigwedge^2V^\ast \otimes \mathfrak{so}(V)$ consisting of elements $R\in \bigwedge^2V^\ast\otimes \mathfrak{so}(V)$ 
that  satisfy the algebraic Bianchi identity, i.e., 
$$ R(w_1, w_2) w_3 +R(w_2, w_3) w_1+R(w_3, w_1) w_2=0,$$
The elements of $\mathcal{AC}(V)$   are called \emph{algebraic curvature (operators)} on $V$.  It is known that   $\dim \mathcal {AC} (V) =  \frac{1}{12}(\dim V)^2 ((\dim V) ^2 -1)$ \cite[Corollary 1.8.4, p. 45]{Gilkey2001}. 

$\bullet$  It is known  that  the  image  $R^{\mathrm{Id}}$ 
of the operator $\mathrm{Id}: \bigwedge ^2   V^*  \to \bigwedge ^ 2 V^*$  in   $\bigwedge ^2 V^* \otimes   \so (V)$ via  the identification $\bigwedge ^2 V^* $ with $\so (V)$     is  an algebraic  curvature  of constant  sectional  curvature, see. e.g.
\cite[Lemma 1.6.4, p. 31]{Gilkey2001}.  It is immediate to  see  that
\begin{equation}\label{eq:rid1}
R^{\rm Id}(w_1, w_2)w_3 := \langle w_2, w_3\rangle w_1 - \langle w_1, w_3\rangle w_2, 
\end{equation}
for any $w_1, w_2, w_3$, see \cite[p. 31]{Gilkey2001}. By the Schur  lemma if $(M,g)$ is a connected Riemannian manifold of dimension at least 3, then the Riemannian curvature tensor of $M$ is of the form $R=\lambda(x) R^{\mathrm{Id}}$ at any point $x\in M$  if and only if $(M,g)$ has constant curvature \cite[Theorem 2.2, p. 202]{KN1963}.

$\bullet$ Let  $Der(\Om^*(M))$  be the   graded Lie algebra of graded derivations of $\Om^* (M)$. 
For $K\in \Om^*(M, TM)$ we denote by $\iota_K$ and by $\mathcal{L}_K=[d,\iota_K]$ the contraction with $K$ and corresponding the Lie derivative, respectively. It is  known that  $\Ll\colon \Om^*(M, TM) \to Der(\Om^*(M))$ is injective,  and moreover \cite{FN1956a, FN1956b}
$$\Ll(\Om^* (M, TM)) = \{  D \in Der (\Om^*(M))|\:  [D, d] = 0\}.$$
Hence  $\Ll(\Om^* (M, TM))$ is  closed  under  the graded  Lie bracket $[, ]$  on $Der(\Om^*(M))$ and one then defines the \emph{Fr\"olicher-Nijenhuis bracket}  $[, ] ^{FN}$ on $\Om^*(M, TM)$   as  the pull-back  of the   graded Lie bracket  on $Der(\Om^*(M))$ via  the  linear embedding $\Ll$, i.e., 
$$\Ll_{[K, L]^{FN}}: = [\Ll_K, \Ll_L].$$

\section[The first integrability condition]{A reformulation of  the first condition  for  the formal integrability  of  CR-twistor   spaces  $({\mathbb{G}}, B. J_{ g, \chi})$}\label{sec:reformulation}
In this section we  reformulate the first condition  for  the  integrability  of  CR-twistor   spaces  $({\mathbb{G}}, B, J_{ g, \chi})$  in terms of a system  of linear  equations for the curvature tensor of $(M,g)$. This will in particular imply that  the  first integrability condition is automatically satisfied in the  case  $ (r,m)=(1,2n)$  or  $ (r, m)= (m-1,m)$. The proof goes in two steps. First we  express the first integrability condition as the condition $[J_{g,\chi}X,J_{g,\chi}Y]^{\mathrm{vert}}=[X,Y]^{\mathrm{vert}}$ for any  $X,Y\in \Gamma(B)$ (Corollary \ref{cor:cond2}). Then we translate this in a system of conditions on the curvature tensor of $(M,g)$  (Proposition \ref{prop:ver32}).
\subsection{The equation $[J_{g,\chi}X,J_{g,\chi}Y]^{\mathrm{vert}}=[X,Y]^{\mathrm{vert}}$}
Let $(M, g, \chi)$  be  a Riemannian manifold    endowed  with  a  $r$-fold VCP structure  and $(B, J_{g, \chi})$ the
almost CR-structure  on  ${\mathbb{G}}$.
 Let $E^*$ be the dual bundle of the tautological bundle $E$  over ${\mathbb{G}}$. At every point $v$ in ${\mathbb{G}}$, the Riemannian metric $g$ induces a natural isomorphism
 $E_v=T_{\pi(v)}M/E_v^\perp$, where $\pi\colon {\mathbb{G}} \to M$ is the projection to the base. This gives a natural identification
$E^* = \mathrm{Ann} (E^\perp)$, where  $\mathrm{Ann} (E^\perp)$ is the vector bundle over ${\mathbb{G}}$  whose
	fiber over $v$ consists of all elements of $T_{\pi (v)}^*M$ that annihilate $E_v ^\perp$.

\begin{remark}\label{rem:bmap}Since $E^*$ is a subbundle of $\pi^*T^*M$ via the identification $E^* = \mathrm{Ann} (E^\perp)$, 
any  $ \theta\in  \Gamma (E^*)$ defines  a    map  of fiber bundles over $M$,
\begin{equation}\label{eq:bmap}
\hat \theta:  {\mathbb{G}} \to T^*M,  
\end{equation}
mapping a point $v\in {\mathbb{G}}$ to the element $\theta_v$ seen as an element in $T_{\pi (v)}^*M$. 
\end{remark}
The  Riemannian metric $g$   induces a natural  Riemannian metric on   the vector bundle $T^*M\to  M$      and so endows  $T^*M$  with   the associated Levi-Civita  connection $\nabla ^{LC}$ and the corresponding splitting of the tangent bundle of the total space of $T^*M$ into a vertical and a horizontal subbundle. The same applies to the bundle $\pi^*T^*M$ and to its subbundle $E^*$. 
\begin{notation}
 For   $ v\in {\mathbb{G}}$  we let
 \begin{equation}\label{eq:gamma-hor}
 \Gamma _{\hor(v)}  (E^*) : = \{  \theta  \in  \Gamma  (E^*): \;    d  \hat  \theta ( T _v ^{\hor} {\mathbb{G}} ) \subset    T^{\hor}_{ \hat \theta (v)}  T^*M   \}.
 \end{equation}
 \end{notation}
 In other words,  $\Gamma _{\hor(v)}(E^*)$   consists of  elements in    $\Gamma (E^*)$           that   are   ``horizontal" at  $v$.
 Using parallel transport in the Grassmann bundle ${\mathbb{G}}\to M$ and in the total space of the vector bundle $E^*$ on ${\mathbb{G}}$ seen as a fiber bundle over $M$ one easily shows that every vector $\xi\in E^*_v$ can be extended to a section of $E^*$ that is horizontal at $v$. For later reference, we state this fact as the following Lemma.
 \begin{lemma}\label{lem:exist}  For any  $\xi \in  E^*_v$ there exists     an  element    $\theta \in \Gamma_{\hor(v)} (E^*)$ such that  $\theta_v = \xi$.
\end{lemma}

\begin{notation}
We write
\begin{align*}
\eta\colon \Gamma(E^*) &\to \Omega^1(\mathbb{G}),\\
\theta&\mapsto \eta[\theta]
\end{align*}
for the map sending a smooth section $\theta$ of $E^*$ to the 1-form $\eta[\theta]$ given by
\[
 \eta[\theta]_v(w)= \theta_v  (d\pi_v(w)),
\]
 for any $w\in T_v{\mathbb{G}}$. In other words, $\eta[\theta]$ is the section of $T^*\mathbb G \to \mathbb G$ given by the composition
 \[
 \mathbb G \xrightarrow{\theta} E^*\hookrightarrow \pi^*T^*M\xrightarrow{(d\pi)^*} T^*{\mathbb{G}}.
 \]
 \end{notation}

\begin{lemma}\label{lemma:annihilator}
Let us consider the subbundle $B \oplus^\perp   T^{\mathrm{vert}} {\mathbb{G}}$ of $T {\mathbb{G}}$. For any $v\in 
{\mathbb{G}}$ we have
\[
B_v \oplus^\perp   T^{\mathrm{vert}}_v {\mathbb{G}}=\bigcap_{\theta\in \Gamma_{\hor(v)}(E ^*)} \ker(\eta[\theta]_v).
\]
\end{lemma}
\begin{proof}  
Let $w=w_{B}+w_{\mathrm{vert}} \in T_v{\mathbb{G}}$, where $w _B \in B_v$ and $w_{\mathrm{vert}} \in T_v^{\mathrm{vert}} {\mathbb{G}}$. Then for every $\theta\in \Gamma_{\hor(v)}(E^*)$  we have
\begin{equation}\label{eq:theta}
\eta[\theta]_v(w)=\theta_v(d\pi_v(w_{B}))=0.
\end{equation}
Namely, by definition of $B_v$, equation (\ref{eq:B}), the vector $d\pi_v(w_{B})$ is in $E_v^\perp$ and so it is annihilated by $\theta_v\in E^*_v=\mathrm{Ann}(E_v^\perp)$. 
Vice versa, let $w\in  T_v {\mathbb{G}}$ be such that $\eta[\theta]_v(w)=0$ for every $\theta\in \Gamma_{\hor(v)}(E^*)$.  Let us write $w=w_{\mathrm{hor}}+w_{\mathrm{vert}}$, with   $w_{\hor/\mathrm{vert}} \in T^{\hor/\mathrm{vert}} {\mathbb{G}}$. Let $\xi\in E_v^*$. By Lemma \ref{lem:exist}, there exists $\theta \in \Gamma_{\hor(v)} (E^*)$ such that  $\theta_v = \xi$, and so 
 \[
 \xi(d\pi_v(w_\hor))= \xi(d\pi_v(w))=\theta_v(d\pi_v(w))=\eta[\theta]_v(w)=0.
 \]
 Therefore,  
 \[
 d\pi_v(w_\hor)\in \bigcap_{\xi\in \mathrm{Ann}(E_v^\perp)} \ker(\xi)=E_v^\perp,
\]
and so $w_{\mathrm{hor}}\in B_v$. This completes the  proof of Lemma \ref{lemma:annihilator}.
\end{proof}

\begin{lemma}\label{lem:lem21ver} 
		Let $v\in {\mathbb{G}}$. For  any  $\theta \in \Gamma_{\hor(v)}(E^*)$
 one has   
\[		
		(d\eta[\theta])_v\bigr\vert_{\bigwedge^2 T^{\mathrm{hor}} _{v} {\mathbb{G}}  } =0
\]	

\end{lemma}
	\begin{proof}  
	The bundle $E^* =\mathrm  {Ann} (E ^\perp)$ over ${\mathbb{G}}$ is a subbundle of the   bundle $\pi ^* T^*M$  and therefore    a section  $\theta$ of $\Gamma (E^*)$ is a section  of $\pi^* T^*M$.  
 We have a  commutative diagram
    \[
 \xymatrix{
  T^*\mathbb{G}\\
  \pi^*T^*M\ar[u]^{(d\pi)^*}\ar[d]_{\mathrm{pr}}\ar[r]^{\hat{\pi}}& T^* M\ar[d]^{\mathrm{pr}}\\
\mathbb G \ar[r]_{\pi}\ar@/^1.5pc/[u]^{\theta}\ar@/^4.5pc/[uu]^{\eta[\theta]}\ar[ur]^{\hat{\theta}}& M,
 }
  \] 
 where $\hat{\theta}\colon \mathbb{G}\to T^*M$ is the map defined in Remark \ref{rem:bmap},
  and so
  \begin{equation}\label{eq:Liou}
   \eta[\theta]  = 
    \hat \theta^* (\Theta_{\mathrm{Lio};M}),
  \end{equation} 
 where  $\Theta_{\mathrm{Lio};M}$   is the   Liouville 1-form on $T^* M$.
 Let $\om $ be the canonical symplectic  form on $T^*M$.   From \eqref{eq:Liou} we obtain
 
  \begin{equation}\label{eq:omega}
   d\eta[\theta] 
   =  \hat \theta^* (\om)=\om\circ (d\hat{\theta}\wedge d\hat{\theta}).
   \end{equation}
 Since $\theta \in \Gamma_{\hor(v)}(E^*)$, the differential $d\hat \theta$ maps  the horizontal
  space $T_v^\hor \mathbb G$  to  $T^\hor_{\hat{\theta}(v)} T^*M$. It   is well-known that the restriction of the canonical 2-form $\omega$ to $\bigwedge^2 T^\hor T^*M$ identically vanishes. This concludes the proof of Lemma \ref{lem:lem21ver}.
\end{proof}

\begin{lemma}\label{lem:bracket-in-B-plus-Tvert}
For  any $X, Y \in \Gamma (B)$ we have
$$[X,Y]\in  \Gamma(B\oplus^\perp T^{\mathrm{vert}}{\mathbb{G}}).$$
\end{lemma}
\begin{proof}   Let $v$ be a point in ${\mathbb{G}}$. We have to show that $[X,Y]_v\in  B_v\oplus^\perp T_v^{\mathrm{vert}}{\mathbb{G}}$.
By Lemma \ref{lemma:annihilator} this is equivalent to showing that  for every $\theta \in \Gamma _{\hor(v)}(E^*)$
		we have
		$\eta[\theta]_v([X,Y]_v)=0$. 
By the  Cartan formula,
\[
\eta[\theta]_v ([X, Y]_v)=- (d\eta[\theta])_v(X_v,Y_v) + X_v(\eta[\theta](Y))  - Y_v(\eta[\theta](X)).
\]
By definition of $B$, we have $B_v\subset T_v ^{\mathrm{hor}}{\mathbb{G}}$, and so $(d\eta[\theta])_v(X_v,Y_v)=0$, by Lemma \ref{lem:lem21ver}.
By definition of $\Gamma _{\hor(v)}(E^*)$, $\theta$ is in particular an element of $\Gamma(E^*)=\Gamma(\mathrm{Ann}(E^\perp))$. Therefore, for any point $v'$ in ${\mathbb{G}}$ we have
\[
\eta[\theta]_{v'}(X_{v'})=\theta_{v'}  (d\pi_{v'}(X_{v'}))=0,
\]
since $X\in \Gamma(B)$ and so $d\pi_{v'}(X_{v'})\in E_{v'}^\perp$, by the defining equation \eqref{eq:B}. This means that $\eta[\theta](X)$ identically vanish on ${\mathbb{G}}$. By the same argument, also $\eta[\theta](Y)\equiv 0$, and so we have $\eta[\theta]_v ([X, Y]_v)=0$.
\end{proof}

\begin{corollary}\label{cor:cond2}  Let $J_{g,\chi}$ the complex structure on $B$ induced by the VCP $\chi$ (equation \eqref{eq:BJv}). For  any $X, Y \in \Gamma (B)$ we have
\[
[J_{g,\chi}X,J_{g,\chi}Y]-[X,Y]\in  \Gamma(B) \text{ if and only if }[J_{g,\chi}X,J_{g,\chi}Y]^{\mathrm{vert}}=[X,Y]^{\mathrm{vert}}.
\]
\end{corollary}

\begin{proof} By Lemma \ref{lem:bracket-in-B-plus-Tvert}, we have $[X,Y]\in \Gamma(B\oplus T^{\mathrm{vert}}{\mathbb{G}})$.  Since $J_{g,\chi}$ is an vector bundle endomorphism of $B$, we also have $J_{g,\chi}X,J_{g,\chi}Y\in \Gamma(B)$ and so by Lemma \ref{lem:bracket-in-B-plus-Tvert} again, $[J_{g,\chi}X,J_{g,\chi}Y]\in  \Gamma(B\oplus T^{\mathrm{vert}}{\mathbb{G}})$. This gives $[J_{g,\chi}X,J_{g,\chi}Y]-[X,Y]\in  \Gamma(B\oplus T^{\mathrm{vert}}{\mathbb{G}})$ and so $[J_{g,\chi}X,J_{g,\chi}Y]-[X,Y]\in  \Gamma(B)$ if and only if $([J_{g,\chi}X,J_{g,\chi}Y]-[X,Y])^{\mathrm{vert}}=0$.
\end{proof}

\subsection{A curvature  reformulation of the  first CR-integrability condition}\label{sec:curv}

 By Corollary \ref{cor:cond2}, the   first  CR-integrability  condition \eqref{CR1}   is equivalent to
 $$[J_{g,\chi}X,J_{g,\chi}Y]^{\mathrm{vert}}=[X,Y]^{\mathrm{vert}}
 $$
 for any $X,Y\in \Gamma(B)$.
 This latter condition can be conveniently expressed in terms  of the curvature operator  $R$  of the  Levi-Civita  connection on $(M,g)$. %
If $X, Y $ are  horizontal  vector fields  on  ${\mathbb{G}}$ and $v\in \mathbb G$, we have $R_{\pi(v)}(d\pi_v (X_v), d\pi_v (Y_v))   \in  \so ({T_{\pi (v)}M})$ and so a corresponding $\SO(\dim M)$-invariant vertical vector field $\mathfrak{r}_{\pi(v)}(d\pi_v (X_v), d\pi_v (Y_v))$ on the fiber of $\mathbb G\to M$ through $v$. Evaluating this vector field at the point $v$ we obtain a vertical tangent vector $\mathfrak{r}_{\pi(v)}(d\pi_v (X_v), d\pi_v (Y_v))\bigr\vert_v\in T_v^{\mathrm{vert}}{\mathbb{G}}$. It is a standard fact, that can be easily derived from see e.g. \cite[p. 290]{Besse1986} or \cite[p. 89]{KN1963} by noticing that
$
[-,-]^\mathrm{vert}\colon T^\hor \mathbb{G}\otimes T^\hor \mathbb{G} \to T^\mathrm{vert}\mathbb{G}
$
is a tensor,  that
 \begin{equation}
 \label{eq:vertical-to-curvature}
 [X,Y]^{\mathrm{vert}}_v=-\mathfrak{r}_{\pi(v)}(d\pi_v (X_v), d\pi_v (Y_v))\bigr\vert_v.
 \end{equation}

 We identify   $T_v ^\ver \mathbb{G}$  with   $\Hom (E_v, E_v ^\perp)$ and define a  linear embedding  $\epsilon: \Hom(E_v, E_v ^\perp) \to   \so (E_v\oplus E_v^\perp)$
by  extending the following   relations linearly   for  $\xi \in E_v$, $w\in E_v^\perp$, 
 and $X \in E_v \oplus E_v ^\perp$:
\begin{equation}\label{eq:i-ast}\epsilon (\xi ^\sharp \otimes  w )  (X): =  \xi(X) \cdot w -\la w,X\ra \cdot  \xi,
\end{equation}
 where $\xi^\sharp \in E_v^*$   is dual  to $\xi$  w.r.t.  $g\bigr\vert_{E_v}$. We shall  use the shorthand notation $\xi^\sharp \hat \otimes w$   for $\epsilon ( \xi^\sharp \otimes w  )$.
 The decomposition
\begin{equation}\label{eq:ortho}
\so (E_v \oplus E_v^\perp) = \so (E_v) \oplus \so (E_v ^\perp) \oplus \epsilon (\Hom (E_v, E_v ^\perp)
\end{equation}
is    an orthogonal decomposition w.r.t.  the  Killing metric,\footnote{In what follows  we shall often  omit  ``Killing" when  we  talk about a metric on a compact Lie algebra.} see e.g. \cite[Theorem  1.1, p. 231]{Helgason1978}.
 Let 
$
\Pi_{E_v^\ast \hat  \otimes E_v ^\perp} \colon \so (E_v \oplus E_v^\perp)\to \epsilon (\Hom (E_v, E_v ^\perp))
$
be the orthogonal projection. Then, under the identification $T_v ^\ver \mathbb{G}=  \Hom (E_v, E_v ^\perp)$, we have
$
\mathfrak{r}_{\pi(v)}(w_1, w_2)\bigr\vert_v=\Pi_{E_v^\ast \hat \otimes E_v ^\perp} R_{\pi(v)}(w_1,w_2)
$
for any $w_1,w_2$ in $T_{\pi(v)}M$. Therefore, equation \eqref{eq:vertical-to-curvature} can be rewritten as
\begin{equation}\label{eq:vertical-to-curvature2} 
 [X,Y]^{\mathrm{vert}}_v=-\Pi_{E_v^\ast \hat \otimes E_v ^\perp} R_{\pi(v)}(d\pi_v (X_v), d\pi_v (Y_v)).
\end{equation}

\begin{lemma}\label{lemma:jinv}
The following are equivalent
\begin{enumerate}
\item For any $v\in \mathbb{G}$ and two vectors $w_1,w_2\in E_v^\perp$  one has	
\begin{equation}\label{eq:jinv}
 R_{\pi(v)} (w_1, w_2)- R_{\pi(v)} (J_{E_v^\perp}w_1, J_{E_v^\perp}w_2)\in \so(E_v)\oplus \so(E_v^\perp) \subset \so (E_v \oplus E_v ^\perp),
\end{equation}
where $J_{E_v^\perp}$ is the complex structure on $E_v^\perp$ defined by \eqref{eq:Jev}.
\item For any $v \in  \mathbb{G}$, any $w_3 \in E_v ^\perp$ and $w_4 \in E_v$ one has
\begin{equation}\label{eq:jinv3a}
[\Pi_{\so(E_v^\perp)}R_{\pi(v)} (w_3, w_4),J_{E_v^\perp}]=0\in \so (E_v ^{\perp}).
\end{equation}

\end{enumerate}
\end{lemma}

\begin{proposition}\label{prop:ver32}   The first  condition  \eqref{CR1}  for the  CR-integrability of $(B,  J_{g, \chi})$ is equivalent to \eqref{eq:jinv} (and so to any of the conditions in Lemma \ref{lemma:jinv}).
\end{proposition}

The proofs of Lemma \ref{lemma:jinv} and Proposition \ref{prop:ver32} are straightforward and therefore omitted. Detailed proofs can be found in  arXiv:2203.04233v2.

\begin{remark}\label{rem:verprob}  In \cite{Verbitsky2011}  Verbitsky also   expresses  the  integrability condition   for the CR-twistor space  over   a  Riemannian $(M^7, g)$  endowed  with  an associative  3-form $\varphi$ in terms  of   constraints on the   curvature of the underlying Riemannian  manifold $(M^7, g)$.   The Condition (ii) in  \cite[Proposition 3.2]{Verbitsky2011}  is  equivalent to our condition \eqref{eq:jinv3a}.  But    his  assertion in \cite[Proposition 3.2]{Verbitsky2011} that  this condition is equivalent  to   the condition that  $R(w_i\wedge  w_j)$ takes value  in   the Lie algebra $\g_2$  is not correct.  In fact,  that assertion  also  contradicts  a   related statement in \cite[Theorem 11.1]{SW2017}. 
\end{remark}

\begin{examples}\label{ex:CR1LeBrun} (1) In the case $(r,m)=(1,2n)$, the vector $w_4$ in  equation \eqref{eq:jinv3a} is neccessarily 0, so  \eqref{eq:jinv3a} is trivially satisfied.
	
(2)  In the case $(r,m)=(m-1,m)$, the vector space $E_v ^\perp$ is of  real dimension $m-(r-1)= 2$. Therefore,   $\so  (E_v ^\perp)$ is an abelian Lie algebra and the  second condition in Lemma \ref{lemma:jinv} is trivially satisfied. 

(3) Let $E_w$ be the oriented 2-plane in $T_{\pi(v)}M$ spanned by the ordered basis $(w_3,w_4)$. If $R(w_3,w_4)\in   \so (E_{w})\subset \so(T_x M)$ for any $w_3,w_4\in T_{\pi(v)}M$, then \eqref{eq:jinv3a}  is automatically satisfied. In the later part  of this section,  we will see that for 2-fold vector cross products on 7-dimensional manifolds and for 3-fold vector cross products on 8-dimensional manifolds the condition $R_{\pi(v)}(w_3,w_4)\in   \so (E_{w})$ for the Riemannian curvature is also necessary. 
\end{examples}

It   follows  from     Example \ref{ex:CR1LeBrun}  and Proposition \ref{prop:ver32} that   the    Condition \eqref{CR1} is non-trivial only for  two cases  $(r, m) = (2, 7)$ and    $(r, m)  = (3, 8)$.

\subsection{An infinite system of linear conditions for $R$}

Equation \eqref{eq:jinv3a} can be interpreted as a system of linear conditions for a section of a certain vector bundle over $M$. 

Let $V$ be an Euclidean space  endowed with an $r$-fold VCP.
For any $w\in \mathbb{Gr} ^+ (2, V)$   let
\begin{equation}\label{eq:CRcommuten}
\mathcal{R}_w := \{ A_w \in  \so (V)|\, [\Pi_{\so(E_v^\perp)}A_w,J_{E_v^\perp}]=0\}
\end{equation}
for any $v\in \mathbb{Gr} ^+ (r-1, V) $ with $\dim(E_v\cap E_w)=1$  and $\dim(E_v^\perp\cap E_w)=1$.

The following Lemma is immediate from the definition of the subspaces $\mathcal{R}_w$  and   Proposition \ref{prop:ver32}.
\begin{lemma}\label{lem:subrepresentation}      The first  condition  \eqref{CR1}  for the CR-integrability  of  the  CR-twistor  space $({\mathbb{G}}, B, J_{g,\chi})$  over  a manifold  $(M, g, \chi)$ holds if and only  if    for any $ w \in   \mathbb{Gr} ^+ (2, TM)$ we have
	$$R(w) \in  \Rr_w.$$

\end{lemma}

\begin{remark}\label{rem:CR-via-Stiefel}
The condition $\dim(E_v\cap E_w)=1$  and $\dim(E_v^\perp\cap E_w)=1$ means that there exists an orthonormal frame $(w_1,\dots,w_r)$ with $(w_2,\dots, w_{r})$ an orthonormal basis for $E_v$ and $(w_1,w_2)$ an orthonormal basis for $E_w$. Therefore, the first integrability   condition \eqref{CR1}  holds for $(M, g, \chi)$ if and only if   for any  $x \in M$,  and any   orthonormal frame $(w_1,\dots,w_r)$ in $T_xM$  with $(w_2,\dots, w_{r})$ an orthonormal basis for $E_v$ and $(w_1,w_2)$ an orthonormal basis for $E_w$ we have
	\begin{equation}\label{eq:CRcommute8}
	[\Pi_{\so(E_{w_2\wedge\cdots\wedge w_{r}}^\perp)}R_g(x;w_1\wedge w_2),J_{E_{w_2\wedge\cdots\wedge w_{r}}^\perp}]=0 \in \so (E_{w_2\wedge\cdots\wedge w_{r}}^\perp),
	\end{equation}
where $R_g(x;-)$ denotes the curvature tensor of $(M,g)$ at the point $x$.
\end{remark}

\begin{definition} Let $(V,\langle,\rangle)$ be an Euclidean vector space endowed with an $r$-fold VCP $\chi$. We denote by $\mathcal{AC}_{CR1}(V,\chi)\subseteq \mathcal{AC}(V)$ the subspace of 
$\mathcal{AC}(V)$ consisting of those elements $R\in \mathcal{AC}(V)$ such that $R(w)\in \mathcal{R}_w$, for any $w\in \mathbb{Gr} ^+ (2, V)$,
i.e., such that \eqref{eq:CRcommute8} holds for any   orthonormal frame $(w_1,\dots,w_r)$ in $V$.

\end{definition}

\begin{remark}The conditions defining the subspace $\mathcal{AC}_{CR1}(V,\chi)$ of $\mathcal{AC}(V)$ are an infinite system of linear equations. This is the infinite system the title of this Section alludes to.
\end{remark}

\begin{lemma}\label{lem:rid}
We have $R^{\rm Id} \in \mathcal{AC}_{CR1}(V,\chi)$.
\end{lemma}
\begin{proof}
For any orthonormal $r$-frame $w$ one has $\Pi_{\so(E_{w_2\wedge\cdots\wedge w_{r}}^\perp)}R^{\mathrm{Id}}_{w_1\wedge w_2}=0$. Indeed, this identity is equivalent to the condition
\[
\langle R^{\mathrm{Id}}(w_1,w_2)z_1,z_2\rangle =0, \qquad \forall z_1,z_2\in E_{w_2\wedge\cdots\wedge w_{r}}^\perp
\]
which in turn is immediate from the definition of $R^{\mathrm{Id}}$ (see Notation and Conventions).
\end{proof}

\begin{corollary}
If $\dim V\geq 2$ then $\dim \mathcal{AC}_{CR1}(V,\chi)\geq 1$.
\end{corollary}

\begin{remark}\label{rem:montecarlo}
If $I$ is a set of $N$ orthonormal $r$-frames in $V$ then we can consider the set
\[
\mathcal{AC}^{[I]}_{CR1}(V,\chi)=\{R\in \mathcal{AC}(V) \,|\,  [\Pi_{\so(E_{w_2\wedge\cdots\wedge w_{r}}^\perp)}R_{w_1\wedge w_2},J_{E_{w_2\wedge\cdots\wedge w_{r}}^\perp}]=0, \forall w\in I\}.
\]
Clearly, for any $I$ one has $\mathcal{AC}_{CR1}(V,\chi)\subseteq \mathcal{AC}^{[I]}_{CR1}(V,\chi)$ and so if $\dim V\geq 2$ then for any $I$ one has $1\leq \dim \mathcal{AC}_{CR1}(V,\chi)\leq \dim \mathcal{AC}^{[I]}_{CR1}(V,\chi)$.
This paves the way to determining $\dim \mathcal{AC}_{CR1}(V,\chi)$ via Monte Carlo methods: one randomly picks a finite subset $I$ and computes the corresponding dimension of $\mathcal{AC}^{[I]}_{CR1}(V,\chi)$. If this happens to be equal to 1, then one sees that necessarily $\dim \mathcal{AC}_{CR1}(V,\chi)=1$.
 \end{remark}

\begin{proposition}\label{prop:cr17}
Let $V$ be a 7-dimensional Euclidean vector space endowed with a 2-fold VCP $\chi$. Then  $\dim \mathcal{AC}_{CR1}(V,\chi)=1$.  In particular, $\mathcal{AC}_{CR1}(V,\chi)$ is spanned by $R^{\rm Id}$.
\end{proposition}
\begin{proof}
The pair $(V,\chi)$ can be identified with the 7-dimensional space $\mathrm{Im}\O$ of imaginary octonions  endowed with their standard VCP. In this model one can easily implement a Monte Carlo computation of $\dim \mathcal{AC}_{CR1}(V,\chi)=1$ as descrbed in Remark \ref{rem:montecarlo}. Implementation shows that already with 100 random points one generally obtains  $\dim \mathcal{AC}_{CR1}(V,\chi)=1$. A {\tt sagemath} code implementing this computation is provided and commented in the Appendix. It runs in about 50 minutes on a 2.4 Ghz 8core.
\end{proof}

Denote by $\times$ the 2-fold  VCP on $\Im \O$, see \eqref{eq:vcp2}. Proposition  \ref{prop:cr17} implies the following    corollary immediately
\begin{corollary}\label{cor:cr17}  If $R \in \mathcal{AC}_{CR1}(\Im \O, \times)$ then  $R(w) \in \so (E_w)$ for any  $ w \in \mathbb{Gr}^+  (2, \Im \O)$.
\end{corollary}

\section{Proof of Theorem \ref{thm:main}(1--3)}\label{sec:spin7}

 In this section we     define    the   torsion  tensor $T\in \Gamma (\bigwedge^2 B^* \otimes T{\mathbb{G}})$   on  the total space  ${\mathbb{G}}$  over  a  Riemannian manifold  $(M, g)$ endowed  with a    VCP.  Then, using the  results in the previous section,  we  give  a  proof of  Theorem  \ref{thm:main} (1--2) and  of  Theorem \ref{thm:main} (3) for the (2,7) case.
To prove Theorem \ref{thm:main} (3) for the (3,8) case
we  reduce  the  Equation      \eqref{eq:jinv3a}  for  the case $(r, m) =  (3, 8)$  to the  case     $(r, m) = (2, 7)$    and  utilize the   symmetry of  the equation  \eqref{CR1} as well as  ad-hoc techniques.

  Let  $({\mathbb{G}}, B, J_{g,\chi})$  be   the  CR twistor  space over   a  manifold $(M, g)$ endowed with a   VCP   $\chi$.  
  Define a section
$$ T: {\mathbb{G}}\to \bigwedge ^2  B^* \otimes   T {\mathbb{G}}$$ 
by
\begin{align}
T_v(X, Y) ^\ver  =   ([J_{g,\chi}X,J_{g,\chi}Y]-[X,Y])_v ^ \ver,\label{eq:tver}\\
  T_v(X, Y) ^\hor =  \Pi_B ([J_{g,\chi}X,J_{g,\chi}Y]-[X,Y])-J_{g,\chi}\circ \Pi _B([X,J_{g,\chi}Y]+[J_{g, \chi}X,Y])_v.\nonumber\\
  \label{eq:thor}
  \end{align}
for any $v \in {\mathbb{G}}$ and any $X, Y \in   \Gamma (B)$. By \eqref{eq:vertical-to-curvature},  the RHS of  \eqref{eq:tver}  depends only on $X(v), Y(v)$,  thus  $T_v ^\ver $ is a well-defined  tensor.  We  verify immediately, or utilize   \eqref{eq:FNCR2} below, to conclude that that   the RHS  of  \eqref{eq:thor}, like the Nijenhuis tenor,   depends only  on the value $X(v), Y(v)$.  
Thus  $T$ is a tensor   on   the manifold   ${\mathbb{G}}$. Now,

\begin{itemize}
\item Theorem \ref{thm:main}(1) follows  immediately from   Corollary \ref{cor:cond2}, taking into account \eqref{eq:tver}.
\item Theorem \ref{thm:main}(2) is Example \ref{ex:CR1LeBrun} (1--2).
\item Theorem  \ref{thm:main}(3) for the (2,7) case follows from  Proposition \ref{prop:cr17},  Lemma \ref{lem:rid}, \eqref{eq:tver}, noting that if $\dim M\geq 3$ then a metric $g$ on $M$ which satisfies the condition  $R_g(x) =  \lambda  (x)R_{\Id}$ is a constant   curvature metric by the Schur  lemma.
\end{itemize}

To conclude the  proof  of Theorem \ref{thm:main} (3), i.e., to prove it for the (3,8) case, we need a preparatory result.

Let $\chi$ be the 3-fold  VCP on $\R^8 = \O$, see  \eqref{eq:vcp3}  and $\times$ the  2-fold  VCP on $\R^7 = \Im \O$, and let
$$t\colon \bigwedge^2 \O^*\otimes \so (\O)\to\bigwedge^2(\Im \O)^*\otimes \so (\Im\O)$$ be the restriction/projection operator defined by
\[
t(\alpha\otimes  A)=\alpha\bigr\vert_{\Im\O}\otimes \Pi_{\so(\Im\O)}A.
\]
It induces by restriction a map
	\[
	t\colon \mathcal{AC}_{CR1}(\O,\chi)\to \mathcal{AC}_{CR1}(\Im\O,\times).
	\]

\begin{proposition}\label{prop:red2}   Assume  that $R \in \mathcal{AC}_{CR1}(\O, \chi)$. 
Then for any  $w_1 \wedge w_2 \in  \mathbb{Gr} ^+ (2, \O)$  we have
$$R(w_1 \wedge w_2) \in \so  (E_{w_1 \wedge w_2}).$$
\end{proposition}

\begin{proof} 
Since  $\Spin(7)$ acts  transitively  on   $\mathbb{Gr}^+ (2, \O)$  and the  space  $\mathcal{AC}_{CR1}(\O, \chi)$  is   invariant under  the  $\Spin(7)$-action, it suffices  to  prove   Proposition \ref{prop:red2} for  $ w_1 = i$ and $w_2 = j$.
Since $t (R) \in \mathcal{AC}_{CR1}(\O, \times)$, taking into account Corollary \ref{cor:cr17},  we have $R(i\wedge j)\in \Pi_{\so(\Im\O)}^{-1}(\so(E_{i\wedge j}))$ and so
\begin{equation}\label{eq:redij}
R(i\wedge j)   =  R_1 (i\wedge j) + R_2 (i\wedge j),
\end{equation}
with  $R_1(i\wedge j)$ in $\so (E_{i\wedge j})$ and  $ R_2(i\wedge j)$ in $\epsilon (\Hom (E_1,   \Im  \O))$.  
We   shall show that  $R_2(i \wedge j) = 0$.
Let $\mathcal{U}$ be the $\Stab_{\Spin (7)} (i \wedge j)$-invariant subspace of $\so (\O)$ defined by
\begin{equation}\label{eq:eq-for-U}
\mathcal{U}=\{A \in  \so (\O)|\, [\Pi_{\so(E_{j\wedge w}^\perp)}A,J_{E_{j\wedge w}^\perp}]=0, \forall w\in E_{i\wedge j}^\perp\}.
\end{equation}
Since $R \in \mathcal{AC}_{CR1}(\O, \chi)$, we have $R(i\wedge j)\in \mathcal{U}$ and by Example \ref{ex:CR1LeBrun} (3), we have $R_1(i\wedge j)\in \mathcal{U}$. Therefore $R_2(i \wedge j)\in \mathcal{U}$.
  Now   we    write
 $$ R_2 (i \wedge j) = a_i   1^\sharp \hat \otimes i + a_j 1 ^\sharp \hat \otimes  j  +  a_k   1^\sharp \hat \otimes k + a_l 1 ^\sharp\hat \otimes  \ell   + b_i  1 ^\sharp \hat \otimes  \ell i  + b_j 1 ^\sharp \hat \otimes    \ell j +    b_k  1^\sharp \hat \otimes  \ell k,  $$
where recall that   $1 ^\sharp  \in E_1^*$  is dual to $1$ in $E_1$, and $a_i, a_j, a_k, a_l, b_i, b_i, b_k , b_l \in  \R$. 
From \cite[Proposition 2.1, p. 196]{bryant} and other assertions  therein one obtains that the stabilizer $\Stab_{\Spin(7)} ( i\wedge j)$  acts   transitively on the     product of unit spheres $ S^1  (E_{i \wedge j}) \times    S^5 (E_{i \wedge  j}^\perp)$ so we  can assume  without loss of generality that  $a_j = a_l = b_i = b_j = b_k =0$. Picking $w=\ell$ in \eqref{eq:eq-for-U} we obtain
$$ \Pi_{\so  (E_{j\wedge \ell}^\perp)}R_2(i\wedge j)=
\Pi_{\so  (E_{j\wedge \ell}^{\perp})} (a_i 1^\sharp \hat \otimes i  + a_k   1^\sharp \hat \otimes k  )  = a_i 1^\sharp \hat \otimes i  + a_k   1^\sharp \hat \otimes k .$$
Using the      explicit  form of the Cayley    4-form $\varphi_\chi$ as given in \cite[(4.1)]{KLS2018}
\begin{align*}\varphi_\chi= & e ^{0123}+ e ^{0145} + e^{0167} + e^{0246} - e^{0257}- e^{0347} - e^{0356}\\
& +  e^{4567} + e^{2367} + e^{2345} + e^{
	1357} - e^{1346} - e^{1256} - e^{1247},
 \end{align*}
where $e^{abcd}$ is a shorthand notation for $e^a\wedge e^b\wedge e^c\wedge e^d$, and $(e^i)$ is the dual basis of the standard orthonormal basis of $\O$, one computes
 $$[J_{E_{j\wedge \ell}^\perp}, (a_i 1 ^\sharp \hat \otimes i  + a_k   1^\sharp \hat \otimes k )] =a_i (\ell j)^\sharp \hat \otimes i  + a_k( \ell j )^\sharp\hat \otimes  k - a_i 1^\sharp \hat \otimes  \ell  k -  a_k  1^\sharp \hat \otimes \ell i.$$
This vanishes if and only if  $ a_i = a_k=0$.  
\end{proof}

Theorem  \ref{thm:main}(3) in the  (3,8) case now follows immediately by noticing that the Riemannian curvature tensor $R$ of a connected Riemannian manifold $(M,g)$ with $\dim M\geq 3$ satisfies $R(v,w)\in \so(E_{v\wedge w})$ for any two linearly independent tangent vectors if and only if $(M,g)$ has constant sectional curvature.\footnote{At least the  ``if"  assertion   seems  well known,  see  e.g. \cite[p. 31]{Gilkey2001} for an equivalent formulation, which we  also utilize  below.  The ``only if'' part is an easy consequence of Schur's lemma for the Ricci tensor. A detailed  proof can be found  in arXiv:2203.04233v2.}

\begin{remark}\label{rem;existence}  Given   Riemannian     manifolds  $(M^7, g)$, or $ (M^8, g)$ of constant
	curvature,  the existence  of   a    VCP   product  on $(M^7, g)$ and $(M^8, g)$    is equivalent to the existence  of a section  of the associated $SO(7)/G_2$-bundle  over  $M^7$   and   of the  associated   $SO(8)/\Spin(7)$-bundle over  $M^8$,   respectively.  A section of  the associated   $SO(7)/G_2$-bundle  over $M^7$ exists if and only if the manifold $M^7$ is orientable and spinnable,  i.e., equivalently,  if and only if  the first and the second     Stiefel-Whitney classes $w_1(M^7)$ and $w_2(M^7)$ of $M^7$  vanish, see \cite[Theorem 10.6, Chapter IV]{LM89} or \cite[Proposition 3.2]{FKMS1997}.  A section of  the associated   $SO(8)/\Spin(7)$-bundle  over  $M^8$ exists if and only if $w_1(M^8)=w_2(M^8)=0$ and for  any choice of orientation of $M^8$ one has $p_1(M^8)^2- 4p_2(M^8) \pm 8\chi(M^8)=0$, where $p_1$ and $p_2$ are the first two Pontryagin classes and $\chi$ is the Euler class\cite[Theorem 3.4, Corollary 3.5]{GG1970}. 
	A family  of  $\Sp (2)$-invariant $G_2$-structure   on homogeneous  $7$-sphere   $\Sp(2)/\Sp (1)$ of constant  curvature is given in   Remark  at the end  of   Section 2 in \cite{LMES21}.
	It follows from \cite{ACFR20}  that there  is no homogeneous $\Spin(7)$-structure  on   the
	sphere  $S^8$.
\end{remark}

\section{The second       CR-integrability   condition}\label{sec:second}
In  this section $ (M, g, \chi)$  is a   Riemannian  manifold with  a VCP structure  $\chi$
and   $({\mathbb{G}}, B, J_{ g, \chi})$ is its CR-twistor space.    In Subsection \ref{subs:FN} we    express  the second   integrability condition \eqref{CR2}  in terms  of  the  Fr\"olicher-Nijenhuis tensor and compute  this tensor in terms  of $(M, g, \chi)$  in later subsections. Using this  we  complete  the  proof of the Main Theorem \ref{thm:main}. For  this purpose, we   consider  the natural  metric  $\tilde{g}$ on   the total space $\wedge ^{r-1}  TM$    such  that

(i) for any $v \in \bigwedge^{r-1} TM$,  $T^\ver_v  \bigwedge ^{r-1} TM$ is orthogonal to $T^\hor _v \bigwedge^{r-1} TM$, 

 (ii) for any $ v \in \bigwedge^{r-1} TM$   the restriction   of $\tilde g$  to $T^{\ver}_v \wedge ^{r-1} T_{\pi (v)} M$ coincides with   the     metric on $\bigwedge^ {r-1} T_{\pi (v)} M$ defined  by $g(\pi(v))$,
 
 (iii)  The projection $\pi: (\bigwedge^{r-1}TM, \tilde g) \to (M, g)$  is  a Riemannian  submersion.
 
 If $r =2$  then  $\tilde g$  is the Sasaki  metric  on $TM$ \cite{Sasaki58}.
Abusing notation, we also denote  by $\tilde g$ the restriction of $\tilde  g$ to ${\mathbb{G}}$.
Let   us extend  the  operator $J_{g,\chi}\colon B\to B$  to an operator $\tilde J_B\colon T{\mathbb{G}}\to T{\mathbb{G}}$ on the  whole  space $T{\mathbb{G}}$ by setting
$$ (\tilde J_B)\vert_B = J_{g,\chi},  \qquad  (\tilde   J_B)\vert_{B^\perp} =0,$$
where $B^\perp$ is the orthogonal   complement     to $B$ in $T{\mathbb{G}}$. 

\subsection{The second CR-integrability condition and the  Fr\"olicher-Nijenhuis tensor}\label{subs:FN}

\begin{proposition}\label{prop:CR2}  The   second     CR-integrability  condition   is equivalent to the  following  condition
\begin{equation}\label{eq:FN}
\Pi_B ( [\tilde J_B, \tilde J_B]^{FN}_{ | B} ) = 0.
\end{equation}
\end{proposition}
\begin{proof}   By \cite[Corollary 8.12, p. 73]{KMS1993}, for any  two vector fields $X, Y $ on $T{\mathbb{G}}$ we  have
	
\begin{equation}
\Pi _B ({1\over 2}[\tilde J_B, \tilde  J_B]^{FN} (X, Y))  =\Pi _B ( [\tilde J_B   X,  \tilde  J_B  Y] - [X, Y] - \tilde J_B ( [X, \tilde J_B Y]  + [ \tilde J_B X, Y])).
\label{eq:FNCR2}
\end{equation}
Taking into account $\Pi_B \circ   \tilde J _{B}  =  \tilde J_{B} \circ  \Pi_B$, this   proves  Proposition  \ref{prop:CR2}.
\end{proof}

Now we are  going to express   Condition \eqref{eq:FN} in terms  of   the Levi-Civita covariant  derivative  $\tilde \nabla$  on  the Riemannian manifold  $({\mathbb{G}}, \tilde{g})$.
 Let $\tilde \om$ be the   2-form  on ${\mathbb{G}}$ defined by $\tilde  \om (X, Y) = \tilde g (\tilde  J_B  X, Y)$.
Equivalently, $\tilde \om _v (X,  Y) = \varphi_\chi (v \wedge d\pi_v X \wedge  d\pi_vY)$. In particular, we have
\begin{equation}\label{eq:omega-via-phi}
\tilde \om _v (X,  Y)= (\pi^*\varphi_\chi) (Z,X ,Y)
\end{equation}
for any $Z$ with $d\pi_vZ=v$. Notice that, by construction, $\tilde{\omega}_v$ only depends on the horizontal parts of the tangent vectors $X$ and $Y$ in $T_v{\mathbb{G}}$.

 Denote by  $(e_i)$ an orthonormal   basis  of $T_v  \mathbb{G}$. By \cite[Proposition 2.2]{KLS2018}, we have
	$$ [\tilde J_B , \tilde J_B]^{FN}_{v}  =  2 \sum_{i,j}\big ( (\imath_{e_i}\tilde \om) \wedge  (\imath_{e_j}\tilde{\nabla}_{e_i}\tilde \om) +  \sum_{k}(\imath_{e_j}\imath_{e_i}\tilde \om) \wedge  e^k \wedge (\imath_{e_i}\tilde{\nabla} _{e_k}\tilde \om) \big) \otimes  e_j. $$
Let $m=\dim M$ and $N=\dim \mathbb{G}$. We can choose  $(e_i)$  in such a way  that $e_1, \cdots,  e_{m-r+1}$   is  a basis of $B(v)$. 
With such a choice one has that $\Pi_B( [\tilde J_B , \tilde J_B]^{FN}_{|B(v)} ) = 0$   if and only if for all $ j,p,q \in [1, m+r -1]$ one has
\begin{align*}
\imath_{e_p}\imath_{e_q} \left(\sum_{i\in [1,m-r+1]}(\imath_{e_i}\tilde \om) \wedge  (\imath_{e_j}\tilde{\nabla}_{e_i}\tilde \om) +\sum_{i,k\in [1,m-r+1]}  (\imath_{e_j}\imath_{e_i}\tilde \om) \wedge  e^k \wedge (\imath_{e_i}\tilde{\nabla} _{e_k}\tilde \om)\right)= 0.
\end{align*}
We can choose $(e_1, \cdots,  e_{m-r+1})$ to be
a unitary  frame with respect   to    the  pair $( \tilde{g}\vert_{B(v)}, J_{g,\chi})$, i.e.,  in such a way that $e_{\frac{m+r -1}{2}+k}=J_{g,\chi}e_k$, for $k\in[1, \frac{m+r -1}{2}]$.  The    vectors   $(e_1, \cdots,  e_{\frac{m+r-1}{2}})$  will be called a {\it Hermitian  basis}. 
 With this choice, for $a,b\in [1, m+r -1]$ with $a<b$ one has 
\[
\tilde{\omega}(e_a,e_b)=\begin{cases}
1& \text{ if } e_b=J_{g,\chi}  e_a\\
0 &\text{ elsewhere}
\end{cases}
\]

The second  CR-intergrabiltiy condition is therefore equivalent to the  system 
\begin{align}
\imath_{e_p}\imath_{e_q}\Big ( \imath_{J_{g,\chi}e_p} \tilde \om \wedge  (\imath_{e _j}\tilde{\nabla}_{J_{g,\chi}e_p}\tilde \om) +\imath_{J_{g,\chi}e_q}\tilde \om \wedge (\imath_{e_j}\tilde{\nabla}_{J_{g,\chi}e_q}\tilde \om)\nonumber\\
  +  \sum_{  k \in [1,m-r +1]}  e^k  \wedge (\imath_{J_{g,\chi}e_j}\tilde{\nabla}_{e_k} \tilde \om )\Big ) = 0  \label{eq:jlr}
\end{align}
for  any $j,p,q \in [1, m-r+1]$.
The term  involving $e^k$ in the  last  sum in LHS of \eqref{eq:jlr} vanishes unless $k \in \{  p, q\}$.  So we can  rewrite \eqref{eq:jlr}  as follows
\begin{align}
-  \imath_{e_q}(\imath_{e_j}\tilde{\nabla}_{J_{g,\chi}e_p}\tilde \om) +\imath_{e_p}(\imath_{e_j}\tilde{\nabla}_{J_{g,\chi}e_q}\tilde \om)
+ \imath_{e_q} (\imath _{J_{g,\chi} e_j} \tilde{\nabla} _{e_p}\tilde \om)  - \imath_{e_p} \imath_{J_{g,\chi}e_j} \tilde{\nabla} _{e_q}\tilde \om) =0\nonumber\\
 \text{ for any }  j,p,q \in [1, m -r +1] .\label{eq:pqrgen}
\end{align}

We can now complete the proof of Theorem \ref{thm:main}. 
\begin{itemize}
\item Theorem \ref{thm:main}(4)   follows  from Lemma \ref{lem:bracket-in-B-plus-Tvert}, Proposition \ref{prop:CR2}  and \eqref{eq:thor}.
\item Theorem \ref{thm:main}(5--6) is classical, see Remark \ref{rem:parts-5-6}.
\item A proof of Theorem \ref{thm:main}(7) is the content of the following two Sections.
\end{itemize}
 
\subsection{The  second    CR-integrability   condition for a  $7$-manifold $(M,g)$ with   a   VCP structure}
Since the complex structure $J_{g,\chi}$ on $B(v)$ is given by the vector cross product with the unit vector $v$, and the 2-form $\tilde{\omega}$ is defined in terms of the 3-form $\varphi_\chi$, one should expect that the expressions appearing in \eqref{eq:pqrgen} can be written in terms of $\varphi_\chi$ and of $v\times-$. This is precisely the content of the following Lemma.
\begin{lemma}\label{lemma:to-prove-pqrgen}
Let $x$ be a point in $M$ and let $(v,w_1,w_2,w_3)$ be a orthonormal quadruple in $T_xM$ such that $(w_1,w_2,w_3)$ is a Hermitian basis of $(E_v^\perp, J_{E_v^\perp})$. 
Let $w_{3+k}=v\times  w_k$ for $k\in[1,3]$.
Let  $(e_1,\dots,e_6)$ be an orthonormal basis in $B(v)$ with $w_a=d\pi_v(e_{a})$ for $a\in[1,6]$.
Then the following identities hold  for  $j, p, q \in [1, 6]$
\begin{align}
\label{eq:from-e-to-w-in} \imath_{e_q}(\imath_{e_j}\tilde{\nabla}_{J_{g,\chi}e_p}\tilde \om)&=(\nabla_{   v\times  w_p} \varphi_\chi) (v, w_j, w_q)\\
\label{eq:from-e-to-w-iin} \imath_{e_q} (\imath _{J_{g,\chi} e_j} \tilde{\nabla} _{e_p}\tilde \om)&=(\nabla_{ w_p} \varphi_\chi) ( v,  v \times  w_j, w_q)
\end{align}

\end{lemma}
In order to prove Lemma \ref{lemma:to-prove-pqrgen} we will need some preparation. Let $v\in \mathbb{G}=\mathbb{Gr}^+_1(1,M)\subset TM$. For a tangent vector $X$ in $T_{\pi(v)}M$, we denote by $X^{\mathrm{h.l.}}\vert_v$ and by $X^{\mathrm{v.l.}}\vert_v$ the horizontal and vertical lifts of $X$ to horizontal and tangent vectors in $T_vTM$. Since $T_v^{\mathrm{hor}}TM=T_v^{\mathrm{hor}}\mathbb{G}$, for any horizontal tangent vector $Y\in T_v^{\mathrm{hor}}\mathbb{G}$ we have $Y=(d\pi_v(Y))^{\mathrm{h.l.}}\vert_v$.
\begin{lemma}\label{lemma:gs-prop-7-2}
Let $v\in \mathbb{G}$ and let $(e_i)$ be a orthonormal basis of $T^{\mathrm{hor}}_v\mathbb{G}$ with $d\pi_v(e_7)=v$. Let $w_1,\dots,w_6,v\in T_{\pi(v)}M$ be images of $e_1,\dots,e_7$ via $d\pi_v$, and let $(W_i)$  be the vector fields on a neighborhood $U$ of $\pi(v)$ in $M$ corresponding to the normal coordinates defined by the exponential map $\exp_{\pi(v)}\colon T_{\pi(v)}M\to M$ such that $W_i\vert_{\pi(v)}=w_i$ for $i=1,\dots,6$ and $W_7\vert_{\pi(v)}=v$. Finally let $(\hat{e}_i)$ be the vector fields  in the neighborhood $\pi^{-1}(U)$ of $v$ in $\mathbb{G}$ defined by horizontal lifting of $(W_i)$, i.e., $\hat{e}_i\vert_y=W_{i; \pi(y)}^{\mathrm{h.l}}\vert_y$ for any $y\in \pi^{-1}(U)$.
Then $\hat{e}_i\vert_v=e_i$ and
\[
(\tilde{\nabla}_{\hat{e}_i}\hat{e}_j)^{\mathrm{hor}}\bigr\vert_v=0
\]
for any $i,j\in [1,6]$.
\end{lemma}
\begin{proof}
The first statement is immediate by the definition of $\hat{e}_i$: we have
\[
\hat{e}_i\vert_v=W_{i; \pi(v)}^{\mathrm{h.l}}\vert_v=w_i^{\mathrm{h.l}}\vert_v=(d_{\pi_v}(e_i))^{\mathrm{h.l}}\vert_v=e_i.
\]
For the second statement, let $\nabla$ denote the Levi-Civita connection of $(M,g)$.
Since $\nabla_{W_i}W_j\bigr\vert_{\pi(v)}=0$, by
 \cite[Proposition 7.2 (i)]{GK02} we have
\[
\tilde{\nabla}_{\hat{e}_i}\hat{e}_j\bigr\vert_v=-\frac{1}{2}(R_g(W_i,W_j)\vert_{\pi(v)}v)^{\mathrm{v.l.}}\vert_v,
\]
where $R_g$ denotes the curvature tensor of $(M,g)$. The conclusion immediately follows.
\end{proof}

\begin{lemma}\label{lemma:gs-prop-7-2-bis}
In the same notation as in Lemma \ref{lemma:gs-prop-7-2}, let $Z$ be the vector field on the neighborhood $U$ of $\pi(v)$ in $M$ defined by $Z_x=\mathrm{tra}_{\pi(v),x}(v)$, where $\mathrm{tra}_{\pi(v),x}\colon T_\pi(v)M\to T_xM$ is the parallel transport along the unique geodesics in $U$ from $\pi(v)$ to $x$, and let $\mathcal{H}$ be the horizontal lift of $Z$ to the  neighborhood $\pi^{-1}(U)$ of $v$ in $\mathbb{G}$, i.e.,  $\mathcal{H}_y=Z_{\pi(y)}^{\mathrm{h.l.}}\vert_y$ for any $y\in \mathbb{G}$ with $\pi(y)\in U$. Then $d\pi_v(\mathcal{H}_v)=v$ and
\[
(\tilde{\nabla}_{\hat{e}_i}\mathcal{H})^{\mathrm{hor}}\bigr\vert_v=0
\]
for any $i,j\in [1,6]$.
\end{lemma}
\begin{proof}
The first statement is immediate from $\mathcal{H}_v=Z_{\pi(v)}^{\mathrm{h.l.}}\vert_v=v^{\mathrm{h.l.}}\vert_v$. For the second statement, 
since $Z$ is defined by parallel transport along geodesics stemming from $\pi(v)$ and the vector fields $W_i$ correspond to normal coordinates at $\pi(v)$, we have $\nabla_{W_i}Z\bigr\vert_{\pi(v)}=0$, where $\nabla$ denotes the Levi-Civita connection on $M$. The conclusion then follows from \cite[Proposition 7.2 (i)]{GK02}, by the same reasoning as in the proof of  Lemma  \ref{lemma:gs-prop-7-2}.
\end{proof}

\begin{proof}[Proof of Lemma \ref{lemma:to-prove-pqrgen}]
By assumption, the triple $(e_1,e_2,e_3)$ is a Hermitian basis of $(B(v),  J_{g,\chi})$, so that setting $e_4=J_{g,\chi}e_1$, $e_5=J_{g,\chi}e_2$ and $e_6=J_{g,\chi}e_3$  we obtain a orthonormal basis for  $B(v)$. We complete it to a orthonormal basis for $T^{\mathrm{hor}}_v\mathbb{G}$ by adding a horizontal vector $e_7$ with $d\pi_v(e_7)=v$. 
Let us also write 
\[
u_p=\begin{cases}
w_{p+3}&\text{ if }p\in [1,3]\\
-w_{p-3}&\text{ if }p\in [4,6]
\end{cases};\quad
f_p=\begin{cases}
e_{p+3}&\text{ if }p\in [1,3]\\
-e_{p-3}&\text{ if }p\in [4,6]
\end{cases}
\]
and, accordingly, $\hat{f}_p=\hat{e}_{p+3}$ if $p\in [1,3]$ and $\hat{f}_p=-\hat{e}_{p-3}$ if $p\in [4,6]$.
We are in the assumptions of Lemma \ref{lemma:gs-prop-7-2} and \ref{lemma:gs-prop-7-2-bis} and so, in the same notation as there, we have
$
(\tilde{\nabla}_{\hat{e}_i}\hat{e}_j)^{\mathrm{hor}}\bigr\vert_v= (\tilde{\nabla}_{\hat{e}_i}\mathcal{H})^{\mathrm{hor}}\bigr\vert_v=0.
$
Therefore, recalling that $\tilde{\omega}$ only depends on the horizontal components of its arguments,
\begin{align*}
(\tilde{\nabla}_{\hat{f}_p}(\tilde{\omega}(\hat{e}_j,\hat{e}_q)))\bigr\vert_v&=((\tilde{\nabla}_{\hat{f}_p}\tilde{\omega})(\hat{e}_j,\hat{e}_q))\bigr\vert_v
=(\imath_{\hat{e}_q}\imath_{\hat{e}_j}\tilde{\nabla}_{\hat{f}_p}\tilde{\omega})\bigr\vert_v=\imath_{e_q}\imath_{e_j}\left((\tilde{\nabla}_{\hat{f}_p}\tilde{\omega})\bigr\vert_v\right)\\
&=\imath_{e_q}\imath_{e_j}(\tilde{\nabla}_{f_p}\tilde{\omega}).
\end{align*}
On the other hand,
\[
(\tilde{\nabla}_{\hat{f}_p}(\tilde{\omega}(\hat{e}_j,\hat{e}_q)))\bigr\vert_v=\tilde{\nabla}_{f_p}(\tilde{\omega}(\hat{e}_j,\hat{e}_q))=\frac{d}{dt}\biggr\vert_{t=0}
\tilde{\omega}_{\tilde{\gamma}(t)}(\hat{e}_j\vert_{\tilde{\gamma}(t)},\hat{e}_q\vert_{\tilde{\gamma}(t)}),
\]
where $\tilde{\gamma}$ is any path in $Gr^+1(1,M)$ with $\tilde{\gamma}(0)=v$ and $\frac{d}{dt}\bigr\vert_{t=0}\tilde{\gamma}=f_p$. In particular, we can choose as $\tilde{\gamma}$ the horizontal lift of the geodesics $\gamma_{\pi(v);u_p}$  stemming from the point $\pi(v)$ of $M$ with tangent vector $u_p=d\pi_{v}f_p$. By definition of parallel transport, this lift is $Z_{\gamma_{\pi(v);u_p}}$ so that
\[
(\tilde{\nabla}_{\hat{f}_p}(\tilde{\omega}(\hat{e}_j,\hat{e}_q)))\bigr\vert_v=\frac{d}{dt}\biggr\vert_{t=0}
\tilde{\omega}_{Z_{\gamma_{\pi(v);u_p}(t)}}(\hat{e}_j\vert_{Z_{\gamma_{\pi(v);u_p}(t)}},\hat{e}_q\vert_{Z_{\gamma_{\pi(v);u_p}(t)}}).
\]

By definition of $\mathcal{H}$, we have $d\pi_y\mathcal{H}_y=Z_{\pi(y)}$ for any $y\in \pi^{-1}(U)$, so that, in particular, $d\pi_{Z_{\gamma_{\pi(v);u_p}(t)}}\mathcal{H}_{Z_{\gamma_{\pi(v);u_p}(t)}}=Z_{\gamma_{\pi(v);u_p}(t)}$. By \eqref{eq:omega-via-phi} we therefore have
\begin{align*}
\tilde{\omega}_{Z_{\gamma_{\pi(v);u_p}(t)}}(\hat{e}_j\vert_{Z_{\gamma_{\pi(v);u_p}(t)}},\hat{e}_q\vert_{Z_{\gamma_{\pi(v);u_p}(t)}})&=
(\pi^*\varphi_\chi)(\mathcal{H},\hat{e}_j,\hat{e}_q)\bigr\vert_{Z_{\gamma_{\pi(v);u_p}(t)}}\\
&=\varphi_\chi(Z,W_j,W_q)\bigr\vert_{\gamma_{\pi(v);u_p}(t)},
\end{align*}

and so

\begin{align*}
(\tilde{\nabla}_{\hat{f}_p}(\tilde{\omega}(\hat{e}_j,\hat{e}_q)))\bigr\vert_v&=\frac{d}{dt}\biggr\vert_{t=0}\varphi_\chi(Z,W_j,W_q)\bigr\vert_{\gamma_{\pi(v);u_p}(t)}
=\nabla_{u_p}(\varphi_\chi(Z,W_j,W_q))\\
&=(\nabla_{u_p}\varphi_\chi)(v,w_j,w_q),
\end{align*}

where in the last identity we used the fact that, by construction, $\nabla_{u_p}Z=\nabla_{u_p}W_j=\nabla_{u_p}W_q=0$.
Since $f_p=J_{g,\chi}e_p$ we have
\[
u_p=d\pi_vf_p=d\pi_v(J_{g,\chi}e_p)=J_{E_v^\perp}(d\pi_v e_p)=J_{E_v^\perp}(w_p)=v\times w_p,
\]

so
we have finally found
\[
\imath_{e_q}\imath_{e_j}(\tilde{\nabla}_{J_{g,\chi}e_p}\tilde{\omega})=(\nabla_{v\times w_p}\varphi_\chi)(v,w_j,w_q).
\]

This proves \eqref{eq:from-e-to-w-in}. The proof of \eqref{eq:from-e-to-w-iin} is analogue.
\end{proof}
 From \eqref{eq:pqrgen} and Lemma \ref{lemma:to-prove-pqrgen} we get the following.
\begin{lemma}\label{lemma:cr-2-for-g2}
The second   CR-integrability  holds  for  $(\mathbb{Gr}^+ (1, M^7), B, J_{g, \chi})$ if and only if  for  any       $x \in M,    v  \in \mathbb{G}$ and some (and hence any) Hermitian basis  $w_1,w_2,w_3$ of $(E_v ^\perp,J_{E_v ^\perp})$ the following  conditions  hold
\begin{align}
		(\nabla_{v \times w_q}\varphi_\chi) ( v, w_j ,w_p)  - (\nabla_{v \times w_p}\varphi_\chi) ( v, w_j, w_q)\nonumber \\
		+  (\nabla_{ w_q} \varphi_\chi) ( v,  v \times  w_j, w_p)  - (\nabla_{ w_p}\varphi_\chi) ( v,  v \times w_j, w_q) =0
		\label{eq:with-wn}
		\end{align}
for 	any $j,p,q\in[1,6]$, where $w_{3+k}=v\times w_k$ for $k\in[1,3]$. 
\end{lemma}	
\begin{definition}
Let $V$ be a 7-dimensional Euclidean space endowed with a 2-fold vector cross product $\times$. The space of \emph{algebraic intrinsic torsions} for $(V,\times)$ is the subspace $\mathcal{T}(V)$ of $V^\ast\otimes \bigwedge^3 V^\ast$ consisting of those elements $A$ such that $A(\eta_1;\eta_2\wedge \eta_3\wedge(\eta_2\times \eta_3))=0$ for any $\eta_1,\eta_2,\eta_3\in V$. The $G_2$-invariant subspace $\mathcal{T}_{CR2}(V)$ of $\mathcal{T}(V)$ is defined as the subspace of $\mathcal{T}(V)$ consisting of those elements $A$ such that  for  any  $v  \in \mathbb{Gr}^+ (1, V)$ and some (and hence any) Hermitian basis  $w_1,w_2,w_3$ of $(E_v ^\perp,J_{E_v ^\perp})$ the condition
\begin{align}
		A(v \times w_q; v\wedge w_j \wedge w_p)  - A(v \times w_p;v \wedge w_j \wedge w_q)\nonumber \\
		+  A( w_q; v\wedge (v \times  w_j)\wedge w_p)  - A(w_p; v\wedge  (v \times w_j)\wedge w_q) =0
		\label{eq:with-wn-V}
		\end{align}
holds for 	any $j,p,q\in[1,6]$, where $w_{3+k}=v\times w_k$ for $k\in[1,3]$. 
\end{definition}
\begin{lemma}\label{lemma:dimensions-T}
We have $\dim(\mathcal{T}(V))=49$ and $\dim(\mathcal{T}_{CR2}(V))=0$.
\end{lemma}
\begin{proof}It is well known that $\dim(\mathcal{T}(V))=49$ \cite[Lemma 4.1]{FG1982}. Then \eqref{eq:with-wn-V} imposes an infinite system of linear equations on $\mathcal{T}(V))$, parametrized by quadruples $(v,w_1,w_2,w_3)$ as in the statement of the Lemma. These are obviously not linearly independent. Yet it turns out that it is generally sufficient to sample ten random quadruples to impose 49 linearly independent equations. A {\tt sagemath} code doing this is provided in the Appendix. It runs in about 1 hour on a 2.4 Ghz 8core.

\end{proof}

\begin{remark}\label{rem:fg7}
Fernandez and Gray provide in \cite{FG1982} an explicit decomposition of $\mathcal{T}(V)$ into a orthogonal direct sum of four irreducible $G_2$ representations, of dimensions, 1, 7, 14, 27, respectively. It is likely that working with these and with standard basis quadruples in $(\mathrm{Im}(\mathbb{O}),\times)$ one can obtain $\dim(\mathcal{T}_{CR2}(V))=0$ by imposing by hand a suitable subset of the equations $\eqref{eq:with-wn-V}$ without relying on computer algebra. This would however presumably take much more than an hour to be done.
\end{remark}		

\begin{proposition}\label{prop:par7}  The second   CR-integrability condition holds for a 2-fold VCP structure $(g,\chi)$ on a 7-manifold $M^7$ if and only if the VCP structure $(g,\chi)$ is parallel.
\end{proposition}
\begin{proof}
Immediate from Lemmas \ref{lemma:cr-2-for-g2} and \ref{lemma:dimensions-T}.
\end{proof}

\subsection{The second  CR-integrability  for  a $8$-manifold $(M,g)$ with a VCP structure.}
The analysis of the second  CR-integrability  for  a $8$-manifold $(M,g)$ with a VCP structure goes precisely along the same lines as for the 7-dimensional case. The only change is that now $v$ is an orthonormal frame representing an element in  $\mathbb{G}=\mathbb{Gr}^+ (2, M)$ instead of a unit vecor representing an element in $\mathbb{Gr}^+ (1, M)$. With the same proof, we have the following 8-dimensional analogue of  Lemma \ref{lemma:cr-2-for-g2}.
\begin{proposition}\label{prop:base8}
The second   CR-integrability  holds  for  $(\mathbb{Gr}^+ (2, M^8), B, J_{g ,\chi})$ if and only if  for  any       $x \in M,    v  \in \mathbb{G}$ and some (and hence any) Hermitian basis  $w_1,w_2,w_3$ of $(E_v ^\perp,J_{E_v ^\perp})$ the following  conditions  hold

\begin{align}
(\nabla_{v \times w_q}\varphi_\chi) ( v\wedge w_j \wedge w_p)  - (\nabla_{v \times w_p}\varphi_\chi) ( v \wedge w_j\wedge w_q)\nonumber \\
+  (\nabla_{ w_q} \varphi_\chi) ( v\wedge (v \times  w_j) \wedge w_p)  - (\nabla_{ w_p}\varphi_\chi) ( v\wedge (v \times w_j)\wedge w_q) =0
\label{eq:with-wn8}
\end{align}
for  any   $j,p,q\in[1,6]$, where $w_{3+k}=v\times w_k$ for $k\in[1,3]$.
\end{proposition}

Theorem \ref{thm:main}(7) for the (3,8) case can then be proved by
 the  following     reduction argument.  First  we  recall  that for any $x\in M^8$, one has   $ (\nabla \varphi_\chi)_x \in \mathcal{W}(T_xM) \subset T_xM^* \otimes \bigwedge^4  T_xM^*$, where  
 \begin{align*}
 \mathcal{W}(V)=\{A\in V^*\otimes \bigwedge^4 V^*\,|\, A(\eta_1;\eta_2 \wedge \eta_3 \wedge \eta_4 \wedge&\chi(\eta_2,\eta_3,\eta_4))=0,\,\\
 & \forall \eta_1,\eta_2,\eta_3,\eta_4\in V\},
\end{align*}
see \cite[\S 4]{Fernandez1986}.\footnote{
To emphasize the fact that $(\nabla \varphi_\chi)_x \in \mathcal{W}(T_xM)$, Fernandez calls $\mathcal{W}$ ``the space of covariant derivatives of the fundamental 4-form'' in \cite{Fernandez1986}.}  For any $\xi\in V$, we have a restriction operator
\[
\big\vert_{E_{\xi}^\perp}\colon \mathcal{W}(V)\to \mathcal{T}(E_{\xi}^\perp)
\]
given by
\[
A\vert_{E_{\xi}^\perp} (\eta_1; \eta_2 \wedge \eta_3\wedge \eta_4): = A (\eta_1; \xi \wedge  \eta_2 \wedge \eta_3\wedge \eta_4 ).
\]
Let us denote by $\mathcal{W}_{CR2}(V)$ the subspace of $\mathcal{W}(V)$ consisting of those elements $A$ such that  for  any  $v  \in \mathbb{Gr}^+ (2, V)$ and some (and hence any) Hermitian basis  $w_1,w_2,w_3$ of $(E_v ^\perp,J_{E_v ^\perp})$ the condition
\begin{align}
		A(v \times w_q; v \wedge w_j \wedge w_p)  - A(v \times w_p; v \wedge w_j \wedge w_q)\nonumber \\
		+  A( w_q; v \wedge( v \times  w_j) \wedge w_p)  - A( w_p; v \wedge (v \times w_j) \wedge w_q) =0
		\label{eq:with-wn-V2}
		\end{align}
holds for 	any $j,p,q\in[1,6]$, where $w_{3+k}=v\times w_k$ for $k\in[1,3]$.  It is straightforward  to check  that the restriction operator induces a restriction operator
\[
\big\vert_{E_{\xi}^\perp}\colon \mathcal{W}_{CR2}(V)\to \mathcal{T}_{CR2}(E_{\xi}^\perp).
\]
By Lemma \ref{lemma:dimensions-T} we therefore have that if  $A\in \mathcal{W}_{CR2}(V)$, then $A\vert_{E_{\xi}^\perp}=0$ for any $\xi$, and this means $A=0$.

\section{Conclusions and final remarks}\label{sec:fin}

(1)  In this paper  we unified  and extended  the construction  of a CR-twistor  space  by  LeBrun, Rossi and Verbitsky respectively,  to the case  when the underlying  Riemannian manifold  admits a  VCP structure. We solved the  question of  the formal integrability  of the
CR-structure on the twistor  space,  recovering   the results  by LeBrun and  Rossi respectively,  and   correcting the   result by Verbitsky.

(2)  We expressed  the  formal  integrability  of   a CR-structure  in terms of a torsion  tensor   on the  underlying   space. 
Our method  can be  applied    for    expressing the formal integrability  of   a CR-structure   $(B, J_B)$ on a  smooth manifold $M$ as follows.   First  we  pick up  any  complement  $B^\perp$ of $B$ in $TM$.
Then  we  choose a  metric  $g$ on $M$ such that   (i)  $B ^\perp$ is orthogonal  to $B$,  (ii) $g\vert_B$  is a  Hermitian metric   with respect  to  $J_B$.  Denote by $\Pi _{ B^\perp} $ and $\Pi_B$ the orthogonal  projections  to  $B^\perp$ and $B$ respectively. Define the tensor $T\in \bigwedge^2B^*\otimes TM$ as
\begin{align*}
T(X,Y)=&\underbrace{\Pi_{ B^\perp} ([J_B X, J_B Y] - [X, Y])}_{T^\ver(X,Y)}\\
&\qquad +\underbrace{ \Pi_B ([J_BX,J_BY]-[X,Y])-J_B\circ \Pi _B([X,J_BY]+[J_BX,Y])}_{T^\hor(X,Y)}.
\end{align*}
Then the CR-structure is formally integrable if and only if $T$ vanishes. More precisely,  the first CR-integrability  condition  for $(B, J_B)$ holds  if and only  the tensor  $  T^\ver$  vanishes and the second  CR-integrability condition holds if and only if the tensor $T^\hor$ that vanishes.
 Note that the  integrability  of a CR-structure  has been  investigated   from  the point  of views of  integrability and formal  integrability of  $G$-structures, see \cite[\S 1.6.1, Theorem 1.14]{DT2006}  for a detailed discussion. They also showed that  the intergrability of a CR-structure, viewed as a $G$-structure, implies that the  associated  Levi form vanishes \cite[p. 71]{DT2006}, see also   \cite[p.76]{BHLN20}.  We would like  to mention that   our expression  of the  integrability  of a CR-structure   in terms  of the vertical and horizontal  component of a torsion tensor is  reminiscent  to   O'Brian- Rawnsley's expression of the Nijenhuis  tensor of an almost complex structure on certain  twistor  spaces  in terms of   the  curvature  and torsion  tensor of the associated  connection \cite{OR85}.  

(3)  The characterization of   metric  of constant  curvature on  a Riemannian manifold    of   dimension at  least   3 used at the end of the proof of Theorem \ref{thm:main}(3)
  can be  reformulated in terms  of   representation theory  by saying that, given  $w_0 \in \mathbb{Gr}^+ (2, \R^n)$ and $R \in \mathcal{AC} (\R^n)$,  if   for any  $\gamma \in \SO (n)$, $n \ge  3$
\begin{equation}\label{eq:son}   
  {\gamma}^* R (w_0) \in \so (E_{w_0})
\end{equation}
then $R = \lambda R^{\Id}$ for some $\lambda\in \mathbb{R}$. The whole statement of Theorem \ref{thm:main}(3)
 can  be reformulated  in a similar way: given  $w_0 \in \mathbb{Gr}^+ (2, \R^n)$ and $R \in \mathcal{AC} (\R^n)$,  $n = 7$ or $n = 8$,   if  for  any  $\gamma \in  \G$ or   $\gamma \in \Spin (7)$  respectively,  we have
 \begin{equation}\label{eq:g2spin7}   
 \gamma^* R (w_0) \in  \Rr_{\om_0} \subset \so (\R^n),
 \end{equation}
 where $\Rr_{\om_0}$ is the linear subspace of $\so (\R^n)$ defined by \eqref{eq:CRcommuten},
 then $R = \lambda R^{\Id}$. 
 Using representation  theory, it is not hard to see  that the space  of $\G$-invariant  algebraic  curvatures   on $\R^7$ has dimension 1 and the  space of $\Spin(7)$-invariant  algebraic  curvature on $\R^8$  also has  dimension 1.  Thus the   equations \eqref{eq:g2spin7}  has only   invariant  solutions.  It would be interesting to find a  proof using only representation theory for   Theorem \ref{thm:main}(3) in the (2,7) case.
For this case,  one can compute that   the dimension  of $\Rr_{\om}$ in  the RHS of \eqref{eq:g2spin7}  is  $10$. 

(4)    Like in the (2,7) case, also in the (3,8) case Theorem  \ref{thm:main} (3) can be proved  by directly using   computer algebra, but      this   takes   much longer   machine  time with respect to the 7-dimensional case (approximatively 11 hours on a 2.4 Ghz 8core). One  could     shorten   the    machine  time by investing  human time   on  writing code for  decomposition  of   the space  of   algebraic  curvatures  as $\G$ -and $\Spin(7)$-modules.

(5)  We  gave a new  characterization  of  torsion-free  $\G$ and $\Spin (7)$-structures.  As  we  noted in Remark \ref{rem:fg7}, it would be  interesting  to find a proof of Lemma \ref{lemma:dimensions-T} which would be  only based on representation  theory.


\section*{Appendix: A code for the CR conditions on a 7-dimensional manifold with VCP}

To begin with, we define the command {\tt randomorthogononal2}. It produces a random element in $O(7)$ with rational entries as follows: two random 6-dimensional vectors with entries in $\{0,1\}$ are generated and stereographically projected on $S^7$ so to produce two unit vectors in $\mathbb{R}^7$. Then we produce the orthogonal reflections with respect to the orthogonal hyperplanes to these two vectors and multiply them. Elements in $SO(7)$ obtained as the multiplication of two reflections are general enough so that their first two columns can be an arbitrary orthonormal pair in $\mathbb{R}^7$. The pairs we produce are not uniformly distributed on the Stiefel manifold of pairs of orthonormal vectors in $\mathbb{R}^7$, but for our aims this is not important. For later use we also define the orthogonal projection on the hyperplane orthogonal to a unit vector in $\mathbb{R}^7$. 

\lstdefinestyle{custompy}{
  belowcaptionskip=1\baselineskip,
  breaklines=true,
  frame=L,
  xleftmargin=\parindent,
  language=C,
  showstringspaces=false,
  basicstyle=\footnotesize\ttfamily,
  keywordstyle=\bfseries\color{blue},
  commentstyle=\itshape\color{blue},
  identifierstyle=\color{blue},
  stringstyle=\color{blue},
}

\lstset{escapechar=@,style=custompy}

\begin{lstlisting}
def stereographic7(myvector):
    v=vector(QQ,[0,0,0,0,0,0,0])
    somma=sum(myvector[j]^2 for  j in [0..5])
    denom=1+somma
    v[6]=(-1+somma)/denom
    for i in [0..5]:
        v[i]=2*myvector[i]/denom
    return v
        
def reflection7(myvector):
    M=matrix(QQ,7)
    v=matrix(QQ,1,7)
    for h in [0..6]:
        v[0,h]=myvector[h]  
    for h in [0..6]:
         M[h]=matrix.identity(7)[h]-2*(matrix.identity(7)[h]*transpose(v))*v
         N=transpose(M)
    return N

 
def randomorthogonal2(): 
    A=matrix(QQ,7)   
    s1=[randint(0, 1) for i in [0..5]]
    t1=stereographic7(s1)
    s2=[randint(0, 1) for i in [0..5]]
    t2=stereographic7(s2)
    A=reflection7(t1)*reflection7(t2)   
    return A 
    
def projection7(myvector):
    M=matrix(QQ,7)
    v=matrix(QQ,1,7)
    for h in [0..6]:
        v[0,h]=myvector[h]  
    for h in [0..6]:
         M[h]=matrix.identity(7)[h]-(matrix.identity(7)[h]*transpose(v))*v
         N=transpose(M)
    return N   
\end{lstlisting}

\medskip
We define the 2-fold VCP on $\mathbb{R}^7$ by means of the 3-form
\[
e^{123}+e^{145}+e^{167}+e^{246}-e^{257}-e^{347}-e^{356}
\]
and define the operator $J_v=v\times -\colon \mathbb{R}^7\to \mathbb{R}^7$. When $v$ is a norm one vector, the operator $J_v$ induces a complex multiplication on the hyperplane $v^\perp$. 
\begin{lstlisting}
def unpermutede(myvector):
    a=0
    if myvector==[1,2,3]:
        a=1
    elif myvector==[1,4,5]:
        a=1
    elif myvector==[1,6,7]:
        a=1
    elif myvector==[2,4,6]:
        a=1
    elif myvector==[2,5,7]:
        a=-1
    elif myvector==[3,4,7]:
        a=-1
    elif myvector==[3,5,6]:
        a=-1
    return a
    
def e(myvector):
    a=0
    Pp=[]
    P=Permutations([1,2,3])
    segni=[]
    for p in P:
        segni.append(Permutation(p).sign())
    P=Permutations(myvector)
    for p in P:
        v=[]
        for j in [0..2]:
            v.append(p[j])
        Pp.append(v)       
    a=sum(segni[i]*unpermutede(Pp[i]) for i in [0..5])
    return a
    
def Mult(n):
    M=matrix(ZZ,7,7)
    for i in [1..7]:
        for j in [1..7]:
            if Permutations([n,i,j]).cardinality()==6:
                M[i-1,j-1]=e([n,i,j])
    return M
    
def J(v):
    M=sum(v[i]*Mult(i+1) for i in [0..6])
    return M
\end{lstlisting}
\medskip    
    
Next we impose the algebraic curvature equations on the set of variables $R_{ijkl}$ giving the coefficients of a 4-index tensor in $\mathbb{R}^7$. We collect all these equations in a list of equations named {\tt listone} and we extract the matrix of coefficients of the linear system given by the algebraic curvature equations in the variables $R_{ijkl}$ and call this matrix {\tt MAC}. In order to extract this matrix we use the fact that the coefficient $a_k$ of the variable $x^k$ in the linear equation $a_ix^i$ is $a_k=\partial_{x^k}(a_ix^i)$.    To check we have correctly implemented the algebraic curvature equations we compare the dimension of the space of solution of the algebraic curvature equations obtained from the rank of the matrix {\tt MAC} with the dimension given by \cite[Corollary 1.8.4, p. 45]{Gilkey2001}, finding they match.
    
\begin{lstlisting} 
dim=7
listone=[]
for i in [0..dim-1]:
    for j in [0..dim-1]:
        for k in [0..dim-1]:
            for l in [0..dim-1]:
                listone.append(var('R_%d%d%d%d' % (i,j,k,l))+var('R_%d%d%d%d' % (i,j,l,k)))
                listone.append(var('R_%d%d%d%d' % (i,j,k,l))+var('R_%d%d%d%d' % (j,i,k,l)))
                listone.append(var('R_%d%d%d%d' % (i,j,k,l))+var('R_%d%d%d%d' % (i,k,l,j))+var('R_%d%d%d%d' % (i,l,j,k)))
                listone.append(var('R_%d%d%d%d' % (i,j,k,l))-var('R_%d%d%d%d' % (k,l,i,j)))
                
n=0
MAC = matrix(ZZ,4*7^4,7^4)
for c in listone:
    for i in [0..6]:
        for j in [0..6]:
            for k in [0..6]:
                for l in [0..6]:
                    MAC[n,l+7*k+7^2*j+7^3*i]=derivative(c,var('R_%d%d%d%d' % (i,j,k,l)))
    n=n+1
    
print(7^4-MAC.rank(),1/12*7^2*(7^2-1));
\end{lstlisting}

\medskip
Now we implement the first CR condition for a triple $(w_1,w_2,w_1)$ where $(w_1,w_2)$ is a random orthonormal pair in $\mathrm{Im}(\O)$. To begin with we implement 
the bilinear function $w_1\otimes w_2\mapsto {w_1}^i{w_2}^jR_{ijkl}$. Since the random orthonormal pairs we produce have rational coefficients with respect to the standard basis of $\mathbb{R}^7$, also all of the matrices corresponding to curvature operators, projections and complex multiplications associated with these pairs will have rational coefficients. Since the equations defining the first CR condition are linear in the variables $R_{ijkl}$, it will be computationally convenient to multiply these matrices by a suitable integer to get rid of the denominators. To achieve this we collect the denominator appearing into a vector of rational numbers in the list of integers {\tt denominators}. Finally, we produce a list of 100 random first CR conditions. Each time we pick a random pair $(w_1,w_2)$ of orthonormal unit vectors in $\mathbb{R}^7$, we write the left hand side of the first CR equation $[\Pi_{\so(E_{w_2}^\perp)}R_{w_1\wedge w_2},J_{E_{w_2}^\perp}]=0$ in the form $\Pi_{\so(E_{w_2}^\perp)}R_{w_1\wedge w_2}J_{E_{w_2}^\perp}-J_{E_{w_2}^\perp}\Pi_{\so(E_{w_2}^\perp)}R_{w_1\wedge w_2}$ and multiply it by a suitable nonzero integer to get rid of the denominators. We collect all of these linear equations into a list and extract the matrix of coefficients in the variables $R_{ikl}$ as above. We denote the this matrix by {\tt MCR}. We join the matrices {\tt MAC} and {\tt MCR} into a single matrix {\tt M}. This matrix is the matrix of coefficients of the linear system consisting of the algebraic curvature equations (fixed) together with 100 random first CR equations. Finally, we compute the rank of this matrix to determine the dimension of the space of algebraic curvatures satisfying the 100 random first CR equations. This is an upper bound for the dimension of the space of algebraic curvatures satisfying all of the first CR equations.

\medskip
\begin{lstlisting}   
def ERRE(w1,w2):
    M=matrix(SR,7)
    for k in [0..6]:
        for l in [0..6]:
            M[k,l]=sum(sum(w1[i]*w2[j]*var('R_%d%d%d%d' % (i,j,k,l)) for j in [0..6]) for i in [0..6])
    return M

def denominators(myvector):
    v=matrix(ZZ,[[0,0,0,0,0,0,0]])
    for i in [0..6]:
        v[0,i]=myvector[i].denominator()
    return v

dim=7
sort=100
listone=[]
for p in [1..sort]:
    A=randomorthogonal2()
    w1=A[0]
    w2=A[1]
    d1=lcm(denominators(w1)[0])
    d2=lcm(denominators(w2)[0])
    d=d2^4*d1
    Pro=projection7(w2)
    Rr=ERRE(w1,w2)
    Jay=J(w2)
    M=d*(Pro*Rr*Pro*Jay-Jay*Pro*Rr*Pro)
    for l in [0..dim-1]:
        for h in [0..dim-1]:
            listone.append(M[l,h])
            
n=0
MCR = matrix(ZZ,sort*7^2,7^4)
for c in listone:
    for i in [0..6]:
        for j in [0..6]:
            for k in [0..6]:
                for l in [0..6]:
                    MCR[n,l+7*k+7^2*j+7^3*i]=derivative(c,var('R_%d%d%d%d' % (i,j,k,l)))
    n=n+1                

M=MAC.stack(MCR)

print(7^4-M.rank());
\end{lstlisting}

\medskip
One can write a code solving the first CR condition on $\mathbb{R}^8$ in essentially the same way: one uses the 4-form
\begin{align*}
e^{0123}+e^{0145}&+e^{0167}+e^{0246}-e^{0257}-e^{0347}-e^{0356}\\
&+e^{4567}+e^{2367}+e^{2345}+e^{1357}-e^{1346}-e^{1256}-e^{1247}
\end{align*}
to define the 3-fold VCP on $\mathbb{R}^8$ and uses the multiplication of three reflections to produce random orthonormal triples in $\mathbb{R}^8$.
We omit the details.
 \vskip .7cm
 
Now, we provide a code implementing the second CR condition on $\mathbb{R}^7$. It uses a few of the functions defined in the code implementing the second CR condition that are not repeated here. The strategy is very similar to what we did for the first CR condition. First we impose the antisymmetry conditions on the structure constants $A_{ijkl}$ of an element in $\bigotimes^4 V^*$, where $V=\mathbb{R}^7$, in order to have it be an element in $V^*\otimes \bigwedge^3 V^*$. As a check, at the end we compare the dimension of the space of solutions found this way with the expected dimension of $7\cdot {{7}\choose{3}}=245$.  
\begin{lstlisting}
dim=7
listone2=[]
for i in [0..dim-1]:
    for j in [0..dim-1]:
        for k in [0..dim-1]:
            for l in [0..dim-1]:
                listone2.append(var('A_%d%d%d%d' % (i,j,k,l))+var('A_%d%d%d%d' % (i,k,j,l)))
                listone2.append(var('A_%d%d%d%d' % (i,j,k,l))+var('A_%d%d%d%d' % (i,j,l,k)))


n=0
MACA = matrix(ZZ,4*7^4,7^4)
for c in listone2:
    for i in [0..6]:
        for j in [0..6]:
            for k in [0..6]:
                for l in [0..6]:
                    MACA[n,l+7*k+7^2*j+7^3*i]=derivative(c,var('A_%d%d%d%d' % (i,j,k,l)))
    n=n+1

print(7^4-MACA.rank(),7*7*5);
\end{lstlisting}
We implement 
the multilinear function $w_1\otimes w_2\otimes w_3\otimes w_4\mapsto {w_1}^i{w_2}^j{w_3}^k{w_4}^lA_{ijkl}$, randomly produce 300 orthonormal triples $\eta_1,\eta_2,\eta_3$ in $\mathbb{R}^7$ and for each of these triples we impose the equation $A(\eta_1;\eta_2\wedge\eta_3\wedge(\eta_2\times \eta_3))=0$. As a check, at the end we compare the dimension of the resulting space of solutions with the expected value of 49, the dimension of the space of algebraic intrinsic torsions on $V$.
\begin{lstlisting}
def AAA(w1,w2,w3,w4):
    M=sum(sum(sum(sum(w1[i]*w2[j]*w3[k]*w4[l]*var('A_%d%d%d%d' % (i,j,k,l)) for l in [0..6]) for k in [0..6]) for j in $
    return M  
    
dim=7
sort=300
listaccia=[]
for p in [1..sort]:
    A=randomorthogonal3()
    x=A[0]
    y=A[1]
    z=A[2]
    d1=lcm(denominators(x)[0])
    d2=lcm(denominators(y)[0])
    d3=lcm(denominators(z)[0])
    d=d1*d2^2*d3^2
    Jay=J(y)
    B=d*AAA(x,y,z,Jay*z)
    listaccia.append(B)

n=0
MCRA = matrix(ZZ,sort,7^4)
for c in listaccia:
    for i in [0..6]:
        for j in [0..6]:
            for k in [0..6]:
                for l in [0..6]:
                    MCRA[n,l+7*k+7^2*j+7^3*i]=derivative(c,var('A_%d%d%d%d' % (i,j,k,l)))
    n=n+1

MA=MACA.stack(MCRA)

print(7^4-MA.rank(),49);
    
\end{lstlisting} 
Finally we produce 10 random orthonormal bases $(w_1,w_2,w_3,w_4,w_5,w_6,v)$ with $(w_1,w_2,w_3)$ a Hermitian basis with respect to $J-v=v\times-$. To do this, we first produce a random orthonormal basis $(u_1,u_2,u_3,u_4,u_5,u_6,u_7)$; then using a reflection (if needed) we change the third vector into $u_1\times u_2$, while keeping $u_1$ and $u_2$ fixed. This way we obtain a orthonormal basis of the form $(u_1,u_2,u_1\times u_2,\tilde{u}_4,\tilde{u}_5,\tilde{u}_6,\tilde{u}_7)$. We set $w_1=u_1$, $w_2=u_2$, $w_3=u_1\times u_2$, $v=\tilde{u}_7$ and $w_4=v\times w_1$,  $w_5=v\times w_2$, $w_6=v\times w_3$. Next we implement the equations defining $\mathcal{T}_{CR1}(V)$ for each of these 7-ples and print the dimension of the space of solutions. Since in the first part of the code the reflection matrix was defined only when the input was a norm one vector, we implement at the beginning of the code here its version for arbitrary nonzero vectors. 
\begin{lstlisting}
def reflectionunnorm7(myvector):
    somma=sum(myvector[j]^2 for  j in [0..6])
    M=matrix(QQ,7)
    v=matrix(QQ,1,7)
    for h in [0..6]:
        v[0,h]=myvector[h]
    for h in [0..6]:
         M[h]=matrix.identity(7)[h]-(2/somma)*(matrix.identity(7)[h]*transpose(v))*v
         N=transpose(M)
    return N

def randomg2():
    A=matrix(QQ,7)
    s1=[randint(0, 1) for i in [0..5]]
    t1=stereographic7(s1)
    s2=[randint(0, 1) for i in [0..5]]
    t2=stereographic7(s2)
    s3=[randint(0, 1) for i in [0..5]]
    t3=stereographic7(s3)
    A=reflection7(t1)*reflection7(t2)*reflection7(t3)
    At=transpose(A)
    u1=At[0]
    u2=At[1]
    u3=At[2]
    J1=J(u1)
    u3good=J1*u2
    if u3good==u3:
        B=matrix.identity(7)
    else:
        z=u3good-u3
        B=reflectionunnorm7(z)
    C=transpose(B*A)
    v=C[6]
    D=C
    D[3]=J(v)*C[0]
    D[4]=J(v)*C[1]
    D[5]=J(v)*C[2]
    return D

def randomsimpleg2():
    A=matrix(QQ,7)
    A=matrix.identity(7)
    At=transpose(A)
    u1=At[0]
    u2=At[1]
    u3=At[2]
    J1=J(u1)
    u3good=J1*u2
    if u3good==u3:
        B=matrix.identity(7)
    else:
        z=u3good-u3
        B=reflectionunnorm7(z)
    C=transpose(B*A)
    v=C[6]
    D=C
    D[3]=J(v)*C[0]
    D[4]=J(v)*C[1]
    D[5]=J(v)*C[2]
    return D
    
dim=7
sort=10
listuccia=[]
for p in [1..sort]:
    A=randomg2()
    v=A[0]
    d0=lcm(denominators(v)[0])
    peco=[A[1],A[2],A[3],A[4],A[5],A[6]]
    Jv=J(v)
    for wq in peco:
        for wj in peco:
            for wp in peco:
                d1=lcm(denominators(wq)[0])
                d2=lcm(denominators(wj)[0])
                d3=lcm(denominators(wp)[0])
                d=d0^2*d1*d2*d3
                B=d*(AAA(Jv*wq,v,wj,wp)-AAA(Jv*wp,v,wj,wq)+AAA(wq,v,Jv*wj,wp)-AAA(wp,v,Jv*wj,wq))
                listuccia.append(B)

n=0
MCRAg2 = matrix(ZZ,sort*6^3,7^4)
for c in listuccia:
    for i in [0..6]:
        for j in [0..6]:
            for k in [0..6]:
                for l in [0..6]:
                    MCRAg2[n,l+7*k+7^2*j+7^3*i]=derivative(c,var('A_%d%d%d%d' % (i,j,k,l)))
    n=n+1

MAg2=MA.stack(MCRAg2)

print(7^4-MAg2.rank());    
    
        
    
\end{lstlisting}

\section*{Acknowledgement}  We  would like  to thank Svatopluk Kr\'ysl  for  his discussion on CR-structures  from       Cartan geometry  point of view the Editor and   the Referees  for   helpful comments,  which led us  to a better  result and better exposion of this  paper.

\end{document}